\documentclass[hyperref]{article}
\usepackage{graphicx} 
\usepackage{float} 
\usepackage{amssymb}
\usepackage{amsmath}
\usepackage{mathtools}
\usepackage[thmmarks,amsmath]{ntheorem} 
\usepackage{bm} 
\usepackage[colorlinks,linkcolor=black,citecolor=black]{hyperref}
\usepackage[nameinlink]{cleveref}
\usepackage{extarrows}
\usepackage[hang,flushmargin]{footmisc} 
\usepackage[square,comma,sort&compress,numbers]{natbib} 
\usepackage{mathrsfs} 
\usepackage[font=footnotesize,skip=0pt,textfont=rm,labelfont=rm]{caption,subcaption} 
\usepackage{booktabs} 
\usepackage{tocloft}
\usepackage{listings}
\usepackage{diagbox}
\usepackage{booktabs}
\usepackage{color}
\usepackage{verbatim}

{
	\theoremstyle{nonumberplain}
	\theoremheaderfont{\bfseries}
	\theorembodyfont{\normalfont}
	\theoremsymbol{\mbox{$\Box$}}
	\newtheorem{proof}{Proof.}
}
\usepackage{theorem}
\newtheorem{theorem}{Theorem}[section]
\newtheorem{lemma}{Lemma}[section]
\newtheorem{prop}{Proposition}[section]

\newtheorem{definition}{Definition}[section]

\newtheorem{coro}{Corollary}[section]
\newtheorem{conj}{Conjecture}[section]

\crefname{equation}{Problem}{Problems}
\crefname{section}{Section}{Sections}
\crefname{lemma}{Lemma}{Lemmas}
\crefname{coro}{Corollary}{Corollaries}
\crefname{conj}{Conjecture}{Conjectures}
\crefname{prop}{Proposition}{Propositions}
\crefname{claim}{Claim}{Claims}
\crefname{defn}{Definition}{Definitions}
\crefname{remark}{Remark}{Remarks}
\crefname{exm}{Example}{Examples}
\crefname{theorem}{Theorem}{Theorems}

{
	\theoremheaderfont{\bfseries}
	\theorembodyfont{\normalfont}
	\newtheorem{remark}{Remark}[section]
}
\numberwithin{equation}{section} 
\graphicspath{{figure/}}                                 
\usepackage[a4paper]{geometry}                           
\begin{document}
	\title{\bf\large Bifurcation on Fully Nonlinear Elliptic Equations and Systems\thanks{Supported by National Natural Science Foundation of China (11771428, ????????)} 
	\date{}
	\author{{\small Jing Gao$^1$\thanks{Email: gaojing181@mails.ucas.ac.cn}~~~ Weijun Zhang$^{2}$\thanks{Email: zhangweij3@mail.sysu.edu.cn}~~~Zhitao Zhang$^{1,3,4}$\thanks{Corresponding author. Email: zzt@math.ac.cn}}\\
    {\small $^1$  School of Mathematical Sciences, }\\
    {\small University of Science and Technology of China, Hefei 230041, People's Republic of China,}\\
    {\small $^2$ School of Mathematics, Sun Yat-sen University, Guangzhou 510275, People's Republic of China,}\\
    {\small $^3$ School of Mathematical Sciences, }\\
    {\small University of Chinese Academy of Sciences, Beijing 100049, People's Republic of China,}\\
    {\small $^4$ School of Mathematical Science, Jiangsu University, Zhenjiang 212013, People's Republic of China.}}
}
	\maketitle
	
	\abstract
	{\scriptsize
		In this paper, we study the following fully nonlinear elliptic equations
		\begin{center}	
			$\begin{cases}
				(S_{k}(D^{2}u))^{\frac{1}{k}}=\lambda f(-u)\qquad&in~\Omega,\\
				u=0\qquad&on~\partial\Omega,
			\end{cases}	$
		\end{center} 
	  and coupled systems
    	\begin{center}	
    		$\begin{cases}
    			(S_{k}(D^{2}u))^{\frac{1}{k}}=\lambda g(-u,-v)\qquad&in~\Omega,\\
    			(S_{k}(D^{2}v))^{\frac{1}{k}}=\lambda h(-u,-v)\qquad&in~\Omega,\\
    			u=v=0\qquad&on~\partial\Omega,
    		\end{cases}	$
    	\end{center} 
        dominated by $k$-Hessian operators, where $\Omega$ is a $(k$-$1)$-convex bounded domain in $\mathbb{R}^{N}$, $\lambda$ is a non-negative parameter, $f:\left[0,+\infty\right)\rightarrow\left[0,+\infty\right)$ is a continuous function with zeros only at $0$ and $g,h:\left[0,+\infty\right)\times \left[0,+\infty\right)\rightarrow \left[0,+\infty\right)$ are continuous functions with zeros only at $(\cdot,0)$ and $(0,\cdot)$. We determine the interval of $\lambda$ about the existence, non-existence, uniqueness and multiplicity of $k$-convex solutions to the above problems according to various cases of $f,g,h$, which is a complete supplement to the known results in previous literature. In particular, the above results are also new for Laplacian and Monge-Amp\`ere operators. We mainly use bifurcation theory, a-priori estimates, various maximum principles and technical strategies in the proof.}

    
	\section{Introduction}

The goal of this paper is to study the bifurcation of $k$-Hessian equation and coupled $k$-Hessian system. We mainly focus on the following Dirichlet Problem of $k$-Hessian equation
\begin{equation}
    \label{eq:SK}
	\begin{cases}
		\left(S_{k}(D^{2}u)\right)^{\frac1k}=\lambda f(-u)\qquad&in\quad\Omega\\
			u=0 \qquad&on\quad \partial\Omega
	\end{cases}
\end{equation}	
where $D^{2}u=(\frac{\partial^{2}u}{\partial x_{i}\partial x_{j}})$ is the Hessian matrix of u, $\Omega$ is a bounded $(k$-$1)$-convex domain of $\mathbb{R}^{N}$ with $N\ge 1$, $k=1,2,...,N$, $\lambda$ is a positive parameter and $f:\left[0,+\infty\right)\rightarrow\left[0,+\infty\right)$ is a continuous function with zeros only at $0$. 
 
As a special case, we will also discuss about the following Dirichlet Problem of Monge-Amp\`ere equation
\begin{equation}
    \label{eq:MA}
	\begin{cases}
		(\det(D^{2}u))^{\frac{1}{N}}=\lambda f(-u)\qquad&in\quad\Omega\\
			u=0\qquad&on \quad\partial\Omega
	\end{cases}
\end{equation}
where $\Omega$ is a bounded uniformly convex domain in $\mathbb{R}^{N}$.

We explore the existence, non-existence, uniqueness and multiplicity of the solutions to \cref{eq:SK}. We use a new method so that we do not need the monotonic conditions of $f$ as in \cite{luo_global_2020}. Especially, we get a more general result about the existence, non-existence, uniqueness and multiplicity of Monge-Amp\`ere equation case as \cref{eq:MA}.  

We will also discuss the analogous results for the case of coupled $k$-Hessian systems as followed,
\begin{equation}
	\label{eq:SKS}	
	\begin{cases}
    		(S_{k}(D^{2}u))^\frac1k=\lambda g(-u,-v) \qquad& in\quad\Omega \\  
    		(S_{k}(D^{2}v))^\frac1k=\lambda h(-u,-v)\qquad& in\quad\Omega \\  
    		u=v=0                 \qquad& on\quad \partial\Omega
    \end{cases}
\end{equation} 
where $\Omega$ is a uniformly $(k$-$1)$-convex domain in $\mathbb{R}^{N}$, $\lambda\geq0$, and $g,h:\left[0,+\infty\right)\times\left[0,+\infty\right)\rightarrow\left[0,+\infty\right)$ being continuous, with zeros only at $(0,\cdot)$ and $(\cdot,0)$. 

\subsection{Backgrounds}
Bifurcation theory can explain various phenomena in the natural sciences and physics, for example, when a given physical parameter crossing a threshold, and then it forces the system to organize a new state which differs considerably from that observed before. 

Mathematically speaking, when we study the branch of solutions, if it exists of nonlinear equations, the Implicit Function Theorem tells us if a continuous branch of the stable solutions of nonlinear equations with parameter preserves stability, there is no dramatic change when the parameter is varied. However, if the "ground state" loss its stability when the parameter reaches a critical value, then the system will organizes a new stable state which bifurcating from the ground state. In other words, bifurcation arises from studying equations with parameter to find another solution branching from the known solution branch as the parameter changes. The history and development of bifurcation can be referred to \cite{chow_methods_1982,golubitsky_singularities_1985,golubitsky_singularities_1988,kielhofer_bifurcation_2012,zhang_variational_2013}.

One aspect concerned in the study of bifurcation is the equation with one parameter. For instance, the following model equation
\begin{equation}\label{B1}
    f(x,\lambda)=0.
\end{equation}
There are many mathematicians studying such bifurcation problem of elliptic equations.

The mainly concerned one in this paper is the Rabinowitz global bifurcation of the $k$-Hessian operators. The $k$-Hessian operator can be regarded as the bridge between the linear elliptic equation and the fully nonlinear elliptic equation. Since the $1$-Hessian operator is just the Laplacian and the $N$-Hessian operator is the well-known Monge-Amp\`ere operator. The Hessian equation constitute an important class of fully nonlinear equations. The study of $k$-Hessian equation has many important applications in differential geometry and other related fields. There are a large amount of papers in the literature on the existence, regularity and the qualitative and the qualitative properties of solutions for the $k$-Hessian equation; see \cite{trudinger_dirichlet_1990,trudinger_dirichlet_1995,trudinger_hessian_1997,caffarelli_dirichlet_1985,tso_remarks_1990,wang_class_1994,hou_second_2010} and the references therein.

The first bifurcation result for the Monge-Amp\`ere equation can be found in \cite{jacobsen1999}, in which they considered about
\begin{equation*}
	\begin{cases}
		(\det(D^{2}u))^{\frac1n}=f(\lambda, u)\qquad&in\quad\Omega\\
		u=0\qquad&on \quad\partial\Omega
	\end{cases}
\end{equation*}
for the model as $f(\lambda,u)=|\lambda u|+g(u)$, where $g:\mathbb{R}\rightarrow [0,\infty)$ is continuous and $\lambda$ is a parameter. After that, \cite{zhang_existence_2009} study the bifurcation phenomenon for $f(\lambda,u)=e^{-\frac{\lambda}n u}$ in smooth bounded convex domain, in which they determined the $\lambda$ interval for the existence, uniqueness and nonexistence of the solutions. \cite{dai_eigenvalue_2015} also consider the global bifurcation of radial convex solution and \cite{luo_global_2020} investigate the existence, non-existence of the solutions for almost all cases of $f_0$ and $f_{\infty}$ except the case of $f_{0}=0$ with $f_{\infty}=0$, but they need monotonicity condition about $f$, which could be removed as we can see in \cref{thm:MA2}. As for coupled Monge-Amp\`ere system, \cite{zhang_power-type_2015} study the existence, uniqueness and non-existence of the following power-type coupled system
\begin{equation*}
	\begin{cases}
		\det(D^{2}u)=(-v)^\alpha \qquad&in\quad\Omega\\
  	\det(D^{2}v)=(-u)^\beta \qquad&in\quad\Omega\\
   u<0,~v<0 \qquad&in\quad\Omega\\
		u=v=0\qquad&on \quad\partial\Omega
	\end{cases}
\end{equation*}
where $\Omega$ is a smooth, bounded, strictly convex domain in $R^{N}$, $N\ge 2$, $\alpha>0,~\beta>0$. The detailed results can be found there and the references therein. \cite{qi_nontrivial_2016} also get more existence results for a general system of Monge-Amp\`ere by bifurcation theory.

As for general $k$-Hessian case, the first result is \cite{jacobsen_global_1999}, in which they studied the global bifurcation phenomena of the following equation,
\begin{equation*}
	\begin{cases}
		(S_k(D^{2}u))^{\frac1k}=\vert \lambda u\vert+g(u)\qquad&in\quad\Omega\\
		u=0\qquad&on \quad\partial\Omega
	\end{cases}
\end{equation*}
where $g:\mathbb{R}\rightarrow [0,\infty)$ is continuous and $\lambda$ is a parameter. Furthermore, \cite{dai_bifurcation_2017} study the behavior of global bifurcation continuum of radial admissible solution of \cref{eq:SK} when $\Omega=B$, the unit ball of $\mathbb{R}^{N}$, $N\ge1$. And \cite{dai_global_2018} investigate the global structure of admissible solutions for \cref{eq:SK} in three cases: $f_{0}\in (0,+\infty)$ with $f_{\infty}=0$; $f_{0}=+\infty$ with $f_{\infty}=0$; $f_{0}=0$ with $f_{\infty}\in (0,+\infty)$, we will deal with the remaining cases in \cref{thm:SK1}. 

\subsection{Main Results}

In this paper, we study the existence, non-existence , uniqueness and multiplicity of \cref{eq:SK,eq:MA,eq:SKS} by bifurcation theory, prior estimates, maximum principles and some technical strategies. Employing the a-priori estimate and global bifurcation theory we investigate the global bifurcation phenomena of admissible solution as parameter varies, and then we get the existence of nontrivial admissible solutions. we also study the non-existence, uniqueness and multiplicity by contradiction, which mainly thanks to various maximum principles. 

Denote
\begin{center}
   $f_{0}:=\underset{s\rightarrow 0^{+}}{lim}\frac{f(s)}{s},~f_{\infty}:=\underset{s\rightarrow +\infty}{lim}\frac{f(s)}{s}.$
\end{center}
Recall that $\lambda_{1}$ be the first eigenvalue of $k$-Hessian operator in $\Omega$. 

Our first main results are the following.

\begin{theorem}\label{thm:SK1} 
   Given $f$ satisfying various cases, \cref{eq:SK} must has at least one $k$-admissible solution, provided $\lambda\in$
\begin{table}[!htbp]
\renewcommand{\arraystretch}{1.2}
\centering
\begin{tabular}{|c|c|c|c|}
  \hline
  \diagbox{$f_0$}{$f_\infty$}  & $0$ & $(0,+\infty)$ & $+\infty$ \\
  \hline
  $0$ &  & $(\frac{\lambda_1}{f_\infty},+\infty)$ & $(0,+\infty)$ \\
  \hline
  $(0,+\infty)$ & $(\frac{\lambda_1}{f_0},+\infty)$ & $(\min\{\frac{\lambda_1}{f_0},\frac{\lambda_1}{f_\infty}\},\max\{\frac{\lambda_1}{f_0},\frac{\lambda_1}{f_\infty}\})$ & $(0,\frac{\lambda_1}{f_0})$ \\
  \hline
  $+\infty$ & $(0,+\infty)$ & $(0,\frac{\lambda_1}{f_\infty})$ & \\ 
  \hline
\end{tabular} 
\end{table}
\end{theorem}

The remaining two blanks in the tabular form are listed separately here due to different restrictions on $f$, the results are much more abundant.

\begin{theorem}
\label{thm:SK2} 
    If $f_{0}=+\infty$ and $f_{\infty}=+\infty$, and if $f$ is locally Lipschitz continuous, then there exists a positive constant $\lambda^{*}$ such that \cref{eq:SK} has at least two $k$-admissible solutions for all $\lambda\in (0,\lambda^{*})$, exactly one $k$-admissible solution for $\lambda=\lambda^{*}$ and no nontrivial solutions for all $\lambda>\lambda^{*}$.   
\end{theorem}
 
\begin{theorem}
\label{thm:SK3} 
  If $f_{0}=0$ and $f_{\infty}=0$, furthermore if $f$ is locally Lipschitz continuous and coercive. Then there exists a positive constant $\lambda_{*}$ such that \cref{eq:SK} has at least two $k$-admissible solutions for all $\lambda>\lambda_{*}$, exactly one $k$-admissible solution for $\lambda=\lambda_{*}$ and no nontrivial solutions for all $\lambda\in (0,\lambda_{*})$.  
\end{theorem}

Especially, we can get a general results.
\begin{coro}
    If $f_{\infty}\notin (0,+\infty)$, $f$ is locally Lipschitz continuous and coercive. Let $\lambda_{*}=\underset{(\lambda,u)\in  \mathcal{C}}{inf}\left\{\lambda\right\}$, then for any $\lambda<\lambda_{*}$, \cref{eq:SK} have no nontrivial admissible solution. 
\end{coro}

It is worth mentioning that for $k=N$, that is the Monge-Amp\`ere case, we have better results than \cite{luo_global_2020} as the case of $f_{0}=+\infty$ with $f_{\infty}=+\infty$ without monotonicity hypothesis by using maximum principle to construct contradiction, we only need local Lipschitz continuity of $f$ instead as in \cref{thm:SK2}. Furthermore, we can also prove the non-existence, uniqueness and multiplicity by using convexity as the case of $f_{0}=0$ with $f_{\infty}=0$ without any additional conditions on continuous function $f$ as followed.

\begin{theorem}\label{thm:MA2}
If $f_{0}=0$ and $f_{\infty}=0$, then there exists a positive constant $\lambda_{*}$ such that \cref{eq:MA} has at least two convex solutions for all $\lambda>\lambda_{*}$, exactly one convex solution for $\lambda=\lambda_{*}$ and no nontrivial solutions for all $\lambda\in (0,\lambda_{*})$.   
\end{theorem}

Similar to the proof of \cref{thm:SK2} and \cref{thm:MA2}, we have the following corollary.

\begin{coro}
    Let $\lambda_{*}=\underset{(\lambda,u)\in \mathcal{C}}{inf}\left\{\lambda\right\}$ and $\lambda^{*}=\underset{(\lambda,u)\in \mathcal{C}}{sup}\left\{\lambda\right\}$, if $f$ is locally Lipschitz continuous, then for any $\lambda<\lambda_{*}$ or $\lambda>\lambda^{*}$, \cref{eq:MA} have no nontrivial convex solutions. 
\end{coro}

Inspiring by the results of the Monge-Amp\`ere case, we have the following conjecture on $k$-Hessian case:
\begin{conj} 
If $f_{0}=0$ and $f_{\infty}=0$, there exists a $k^{*}$ such that as $k\ge k^{*}$, there exists a positive constant $\lambda_{*}$ such that \cref{eq:SK} has at least two convex solutions for all $\lambda>\lambda_{*}$, exactly one convex solution for $\lambda=\lambda_{*}$ and no nontrivial solutions for all $\lambda\in (0,\lambda_{*})$.   
\end{conj}

Next, we turn our attentions on the system case, We will establish the preliminary framework for studying couple $k$-Hessian system via bifurcation including eigenvalue problem, the bifurcation of auxiliary equation, a priori estimate and so on. Then we consider the existence, non-existence and multiplicity of the solutions to \cref{eq:SKS}. Up to our knowledge, this is the first paper to study the bifurcation phenomena of couple $k$-Hessian system. 

We take $2$-coupled case as example for convenient, all the following results are still holds in m-coupled $(m\geq2)$ case respectively. 

Denote 
$$g_{0}=\underset{\vert s+t\vert\rightarrow0}{lim}\frac{g(s,t)}{\vert t\vert}, h_{0}=\underset{\vert s+t\vert \rightarrow0}{lim}\frac{h(s,t)}{\vert s\vert},$$
and $$g_{\infty}=\underset{\vert s+t\vert\rightarrow+\infty}{lim}\frac{g(s,t)}{\vert t\vert}, h_{\infty}=\underset{\vert s+t\vert \rightarrow+\infty}{lim}\frac{h(s,t)}{\vert s\vert}.$$
Recall that $\lambda_{1}$ be the first eigenvalue of $k$-Hessian operator in $\Omega$. 

The main results are the following.

\begin{theorem}
\label{thm:SKS1}
Given $g,h$ satisfying various cases, and assume $\mu:=g_0=h_0$, $\nu:=g_\infty=h_\infty$ for convenience, then \cref{eq:SKS} must has at least one $k$-admissible solution $(u,v)$, provided $\lambda\in$
\begin{table}[!htbp]
\renewcommand{\arraystretch}{1.5}
\centering
\begin{tabular}{|c|c|c|c|}
  \hline
  \diagbox{$\mu$}{$\nu$}  & $0$ & $(0,+\infty)$ & $+\infty$ \\
  \hline
  $0$ &  & $(\frac{\lambda_1}{\nu},+\infty)$ & $(0,+\infty)$ \\
  \hline
  $(0,+\infty)$ & $(\frac{\lambda_1}{\mu},+\infty)$ & $(\min\{\frac{\lambda_1}{\mu},\frac{\lambda_1}{\nu}\},\max\{\frac{\lambda_1}{\mu},\frac{\lambda_1}{\nu}\})$ & $(0,\frac{\lambda_1}{\mu})$ \\
  \hline
  $+\infty$ & $(0,+\infty)$ & $(0,\frac{\lambda_1}{\nu})$ & \\ 
  \hline
\end{tabular} 
\end{table}
\end{theorem}

\newpage

Due to different restrictions on $g,h$, the results of the remaining two blanks in the tabular form are listed separately here. Different to \cref{eq:SK}, we need monotonicity assumption about right-hand terms in order to use maximum principle in our proof. 

\begin{theorem}
\label{thm:SKS2}
    If $g_{0}=h_{0}=+\infty$, $g_{\infty}=h_{\infty}=+\infty$, $g$ is non-decreasing with $t$ and $h$ is non-decreasing with $s$. Furthermore, assume $g$ and $h$ are locally Lipschitz continuous. Then, there exists a positive constant $\lambda^{*}$ such that \cref{eq:SKS} has at least two admissible solutions for all $\lambda\in (0,\lambda^{*})$ in $M$, one admissible solution for $\lambda=\lambda^{*}$ in $M$ and no  nontrivial solutions for all $\lambda>\lambda^{*}$.    
\end{theorem}
    
\begin{theorem}
\label{thm:SKS3}
    If $f_{0}=g_{0}=0$, $h_{\infty}:=f_{\infty}=g_{\infty}=0$, $g$ is non-decreasing with $t$ and $h$ is non-decreasing with $s$. Furthermore, assume $g$ and $h$ are coercive respect to $\vert s+t\vert$ and locally Lipschitz continuous. Then there exists a positive constant $\lambda_{*}$ such that \cref{eq:SKS} has at least two admissible solutions for all $\lambda>\lambda_{*}$ in $M$, one admissible solution for $\lambda=\lambda_{*}$ in $M$ and no nontrivial solutions for all $\lambda\in (0,\lambda_{*})$.
\end{theorem}

We can also have a general results as in the equation case.
\begin{coro}
    If $f_{\infty}\notin(0,+\infty)$, $g$ is non-decreasing with  $t$ and $h$ is non-decreasing with $s$. Furthermore, $g$ and $h$ are coercive respect to $\vert s+t\vert$ while they are locally Lipschitz continuous. Let $\underline{\lambda}=\underset{(\lambda,(u,v))\in \mathbb{R}^{+}\times E}{inf}\left\{\lambda:(\lambda,(u,v))\in \mathcal{C}\right\}$, then \cref{eq:SKS} has no solution in $M$ as $\lambda<\underline{\lambda}$.
\end{coro}

This paper is organized as follows. In \cref{sec:pre}, we present some preliminary results for the bifurcation theory. \cref{sec:HE} is devoted to the case of Hessian equations, while \cref{sec:HS} is about the case of Hessian systems.
    
\section{Preliminaries} 
\label{sec:pre}
Firstly, we introduce some basic results on the $k$-Hessian operator.

Let $\Omega$ be a domain in $\mathbb{R}^{N}$, for $k=1,2,...,N$, and $u\in C^{2}(\Omega)$. Recall that the $k$-th elementary symmetric polynomial acting on a $N$-dim vector $\lambda$ as 
\begin{center}
	$S_{k}(\lambda)=\underset{1\le i_{1}<...<i_{k}\le N}{\sum}\underset{j=1}{\overset{k}{\prod}}\lambda_{i_{j}}$.
\end{center} 
The {\bf $k$-Hessian operator} $S_{k}(D^{2}u)$ is defined as
$S_{k}(D^{2}u):=S_{k}(\lambda(D^{2}u))$
where $D^{2}u=(\frac{\partial^{2}u}{\partial x_{i}\partial x_{j}})$ is the hessian matrix of $u$, $\lambda(D^{2}u)=(\lambda_{1},...,\lambda_{N})$ are the eigenvalues of $D^{2}u$. 

In order to introduce the suitable domains to discuss $k$-Hessian operator, we need to introduce the concept of $k$-convexity. Denote
\begin{center}
	$\Gamma_{k}^N=\left\{\lambda\in \mathbb{R}^{N}:S_{j}(\lambda)>0,~\forall j=1,2,...,k\right\}$
\end{center} be the so-called {\bf G\aa{}rding cone} as in \cite{garding_inequality_1959}.
A domain $\Omega$ is called {\bf (uniformly) $k$-convex}, if its principal curvatures $\kappa=(\kappa_{1},...,\kappa_{N-1})$ of $\partial\Omega$ are belonging to $\overline{\Gamma_{k}^{N-1}}$($\Gamma_{k}^{N-1}$).

Then, we say $u\in C^2(\Omega)$ is {\bf(uniformly) $k$-convex} if 
\begin{center}
	$\lambda(D^2u(x))\in \overline{\Gamma_{k}^N}(\Gamma_{k}^N),~\forall x\in \Omega.$
\end{center}
All continuous $k$-convex functions are collected as $\varPhi^k(\Omega)$, and as $\varPhi_0^k(\Omega)$ for those with zero boundary data, which will be our working space. At last, A function $u\in \varPhi_0^k(\Omega)$ solving \cref{eq:SK} is called a {\bf $k$-admissible} solution. 

\begin{remark}
We can extend the concept of $k$-convex into $C^0$ setting as follow. We call a function $u\in C^0(\Omega)$ being $k$-convex, if there exists a sequence $\{u_m\}\in C^2(\Omega)$ such that in any subdomain $\Omega'\Subset\Omega$, $u_m$ is $k$-convex for sufficiently large $m$ and
converges uniformly to $u$. It is easily seen that $u\in C^0(
\Omega)$ is $k$-convex if and only
if $S_k(D^2u)\geq0$ in the viscosity sense, that is, whenever there exists a
point $y\in\Omega$, and a function $v\in C^2(\Omega)$ satisfying $u(y)=v(y), u\leq v$ in $\Omega$, we must
have $S_k(D^2v)\geq0$.
\end{remark}

Let X be the Banach space $C(\overline{\Omega})$ with supremum norm, and let $\Omega$ be a bounded uniformly $(k$-$1)$-convex domain, consider the equation
\begin{equation}
\label{eq:SKG}
	\begin{cases}
		S_{k}(D^{2}u)=\vert g\vert \qquad& in\quad\Omega \\
    		u=0            \qquad& on\quad \partial\Omega
    \end{cases}
\end{equation}
    
\begin{definition} 
Let $g\in X$. For $k=1,2,...,n$, define the solution operator $T_{k}(g)$ by $$T_{k}(g)=u,$$ where $u\in \Phi_0^{k}(\Omega)$ is the $k$-admissible solution to \eqref{eq:SKG}.
\end{definition}

$T_k$ is well-defined ensuring by \cite{trudinger_weak_1997} and completely continuous by \cite{jacobsen_global_1999}.

\vskip 0.2in

We need to use bifurcation theory through out our proofs, we state two classical theorems here for convenience.
    
\begin{lemma}[Global bifurcation \cite{le_global_1997}] \label{lem:GB}
Let $Y$ be a Banach space, $F:\mathbb{R}\times Y\rightarrow Y$ be completely continuous such that $F(\lambda,0)=0$ for all $\lambda\in \mathbb{R}$. Suppose there exist constants $a,b\in \mathbb{R}$ with $a<b$ such that $(a,0),(b,0)$ are not bifurcation points for the equation
    \begin{center}
    	$y-F(\lambda,y)=0$.
    \end{center}
Furthermore, assume for Leray-Schauder degree that
    \begin{center}
    	$deg(id-F(a,.),B_{r}(0))\ne deg(id-F(b,.),B_{r}(0))$,
    \end{center}
where $B_{r}:=(0)\left\{y\in E:\Vert y\Vert <r\right\}$ is an isolating neighborhood of the trivial solution for both constants $a$ and $b$. Let 
    \begin{center}
    $\mathcal{S}=\overline{\left\{(\lambda,y):y-F(\lambda,y)=0,y\ne 0\right\}}\cup \left\{([a,b]\times \left\{0\right\})\right\}$,
    \end{center}
and $\mathcal{C}$ be the component of $\mathcal{S}$ containing $[a,b]\times \left\{0\right\}$. Then either\\
(1) $\mathcal{C}$ is unbounded in $\mathbb{R}\times Y$, or\\
(2) $\mathcal{C}\cap  [(\mathbb{R}\backslash[a,b])\times \left\{0\right\}]\ne \emptyset$.
\end{lemma}
    
\begin{lemma}[Global asymptotic bifurcation \cite{le_global_1997}] \label{lem:GAB} 
Let $Y$ be a Banach space, $F:\mathbb{R}\times Y\rightarrow Y$ be completely continuous such that $F(\lambda,0)=0$ for all $\lambda\in \mathbb{R}$. Suppose there exist constants $a,b\in \mathbb{R}$ with $a<b$ such that the solutions of
    \begin{center}
    	$y-F(\lambda,y)=0$
    \end{center}
are a-priori bounded in $Y$ for $\lambda=a$ and $\lambda=b$; i.e., there exists a constant $M>0$ such that
    \begin{center}
    	$F(a,y)\ne y\ne F(b,y)$,
    \end{center}
for all $y\in Y$ with $\Vert y\Vert \ge M$. Furthermore, assume that
    \begin{center}
    	$deg(id-F(a,.),B_{r}(0))\ne deg(id-F(b,.),B_{r}(0))$,
    \end{center}
for $R>M$. Then there exists at least one continuum $\mathcal{C}$ of solution to $y-F(\lambda,y)=0$ that is unbounded in $[a,b]\times Y$ and either\\
(1) $\mathcal{C}$ is unbounded in the $\lambda$ direction, or else\\
(2) there exists an interval $[c,d]$ such that $(a,b)\cap (c,d)=\emptyset$ and $\mathcal{C}$ bifurcates from infinity in $[c,d]\times Y$. 
\end{lemma}
    
\section{Hessian Equations}
\label{sec:HE}
Now we begin to discuss the following Dirichlet problem of Hessian equation
\begin{equation}
\label{eq:SKE}
	\begin{cases}
		(S_{k}(D^{2}u))^{\frac1k}=\lambda f(-u)\qquad &in \quad\Omega\\
        u=0\qquad&on\quad\partial\Omega
    \end{cases}
\end{equation}
where $\lambda$ is non-negative parameter and $f:\left[0,+\infty\right)\rightarrow\left[0,+\infty\right)$ is a continuous function with zeros only at zero. 

Inspired by \cite{wang_existence_1992}, we have the following a-priori estimate.
\begin{lemma}\label{lem:Apriori}
Consider
\begin{equation}\label{eq:SKAp}
    \begin{cases}
    	S_{k}(D^{2}u)=\Psi(x,u)\qquad&in\quad\Omega\\
    	u=0 \qquad&on\quad\partial\Omega
    \end{cases}
\end{equation}
if there exists a non-increasing function $F(t)\ge 0$ such that 
\begin{equation}\label{eq:GC}
    	\Psi(x,z)\ge F(z) \quad \forall x\in \Omega,~z\le 0 ~~~and ~~~\underset{z\rightarrow-\infty}{lim}F(z)/\vert z\vert ^{k}= +\infty.
\end{equation}
And if $u\in C^{0,1}(\overline{\Omega})$ is an admissible solution of \cref{eq:SKAp}, then $\Vert u\Vert_{0}\leq C$, where $C=C(F,k,\Omega)$.
\end{lemma}
    
\begin{proof} For any fixed $t>0$, denote $\Omega^{t}=\left\{tx:x\in \Omega\right\}$. By the uniqueness of eigenvalues we have $\lambda_{1}(\Omega^{t})=t^{-2}\lambda_{1}(\Omega)$. Without loss of generality we may suppose the origin $O\in \Omega$, which implies $\Omega^{t}\subset \Omega$ for any $t<1$.

If the conclusion of \cref{lem:Apriori} is not true, then there exists a sequence of functions $\Psi(x,z)$ with $\Psi(x,z)\ge F(z),~\forall x\in \Omega,~z\le 0$, so that the sequence of solutions $u_{j}$ of 
\begin{center}
    $S_{k}(D^{2}u_{j})=f_{j}(x,u_{j})\quad in ~\Omega,\qquad u_{j}=0\quad   on~\partial\Omega$
\end{center}
satisfies $M_{j}=\Vert u_{j}\Vert_{0}=-\underset{x\in \Omega}{inf}u_{j}\rightarrow+\infty$. From \eqref{eq:GC} there exist $\overline{\lambda}>\lambda_{1}$ and $t\in (0,1)$ small enough such that $\underset{z\rightarrow-\infty}{lim}F(z)/\vert z\vert ^{k}\ge \overline{\lambda}^{k}$ and $\lambda_{1}(\Omega^{t})=t^{-2}\lambda_{1}(\Omega)>\overline{\lambda}$. For $t$ small enough, with taking a sub-sequence if necessary, we can assume  $u_{j}(x)\rightarrow-\infty$ uniformly in $\Omega^{t}$, hence we have 
\begin{center}
    $S_{k}(D^{2}u_{j})\ge \overline{\lambda}^{k}\vert u_{j}\vert^{k}\qquad in ~\Omega^{t}$
\end{center}
provided $j$ is sufficient large. Fix such a $j$. Let $\phi _{t}$ be the eigen-function of the $k$-Hessian operator on $\Omega^{t}$. Replacing $\phi_{t}$ by $s\phi_{t}$ for some small $s$ we may suppose $\phi_{t}(x)>u_{j}(x)$ in $\Omega^{t}$. Thus $\phi_{t}(x)$ and $u_{j}(x)$ are the super-solution and sub-solution of the Dirichlet problem
\begin{equation}\label{eq:SKAA}
    S_{k}(D^{2}u)=\vert \overline{\lambda}u\vert ^{k}\quad in ~\Omega^{t},\qquad u=0\quad on~\partial\Omega^{t},
\end{equation}
respectively. Therefore there exists a solution u of \cref{eq:SKAA} which satisfies $\phi_{t}(x)\ge u(x)\ge u_{j}(x)$. This means both $\overline{\lambda}$ and $t^{-2}\lambda_{1}$ are the eigenvalues of the $k$-Hessian operator on $\Omega^{t}$, which contradicts the uniqueness of eigenvalues.
\end{proof}

\subsection{Proof of \cref{thm:SK1}}
We will study the existence, non-existence, and multiplicity of $k$-admissible solutions of \cref{eq:SK} via bifurcation method. In order to study \cref{eq:SK}, we need the following auxiliary equation which have studied in \cite{jacobsen_global_1999}
\begin{equation}
\label{eq:SKA}
	\begin{cases}
		S_{k}(D^{2}u)=\lambda^{k} (\vert u\vert^{k}+g(-u))\qquad&in\quad\Omega\\
			u=0\qquad&on\quad \partial\Omega
	\end{cases}
\end{equation} 
where $g:\mathbb{R}^{+}\rightarrow\mathbb{R}^{+}$ is continuous. They proved that $(\lambda_{1},0)$ is a bifurcation point of \cref{eq:SKA}. Basis of this, we get existence, non-existence, and multiplicity of $k$-admissible solutions of \cref{eq:SK} of various $f_{0}$ and $f_{\infty}$. These complete all the classifications on $f$, where in \cite{luo_global_2020} have study this problem for partial results, that we proved the following six situations. Furthermore, we proved a more general result about non-existence such as \cref{thm:SK2}.   

We have the following result:
\begin{theorem}[\cref{thm:SK1}]

(1) If $f_{0},~f_{\infty}\in (0,+\infty)$ with $f_{0}\ne f_{\infty}$, \cref{eq:SK} has at least one admissible solution for every $\lambda\in (min\left\{\frac{\lambda_{1}}{f_{0}},\frac{\lambda_{1}}{f_{\infty}}\right\},max\left\{\frac{\lambda_{1}}{f_{0}},\frac{\lambda_{1}}{f_{\infty}}\right\})$.

(2) If $f_{0}\in (0,+\infty)$ and $ f_{\infty}=0$, \cref{eq:SK} has at least one admissible solution for every $\lambda\in (\frac{\lambda_{1}}{f_{0}},+\infty)$.

(3) If $f_{0}\in (0,+\infty)$ and $ f_{\infty}=+\infty$, \cref{eq:SK} has at least one admissible solution for every $\lambda\in (0,\frac{\lambda_{1}}{f_{0}})$.

(4) If $f_{0}=0$ and $ f_{\infty}\in (0,+\infty)$, \cref{eq:SK} has at least one admissible solution for every $\lambda\in (\frac{\lambda_{1}}{f_{\infty}},+\infty)$.
    
(5) If $f_{0}=0$ and $ f_{\infty} =+\infty$, \cref{eq:SK} has at least one admissible solution for every $\lambda\in (0,+\infty)$.
    
(6) If $f_{0}=+\infty$ and $ f_{\infty}\in (0,+\infty)$, \cref{eq:SK} has at least one admissible solution for every $\lambda\in (0,\frac{\lambda_{1}}{f_{\infty}})$.
    
(7) If $f_{0}=+\infty$ and $ f_{\infty}=0$, \cref{eq:SK} has at least one admissible solution for every $\lambda\in (0,+\infty)$.  
\end{theorem}
   
\begin{proof}
    
The cases (2),(4),(7) have been proved in \cite{luo_global_2020}, so we proceed to show the remaining part. we denote
\begin{center}
    $	f_{0}^{k}:=\underset{s\rightarrow0^{+}}{lim}\frac{f(s)}{s^{k}}$ and $f_{\infty}^{k}:=\underset{s\rightarrow+\infty}{lim}\frac{f(s)}{s^{k}}$.
\end{center}.
    
(1): Let $g:\mathbb{R}^{+}\rightarrow\mathbb{R}^{+}$, $f(s)=f_{0}^{k} s^{k}+g(s)$, $\underset{s\rightarrow0^{+}}{lim}\frac{g(s)}{s^{k}}=0$. Then by \cite{jacobsen_global_1999} there exists an unbounded continuum $\mathcal{C}$ of the set of admissible solution of \cref{eq:SKA} emanating from $(\frac{\lambda_{1}}{f_{0}},0)$ such that $\mathcal{C}\subset(\mathbb{R}\times \Phi^{k}_{0})\cup \left\{(\frac{\lambda_{1}}{f_{0}},0)\right\}$. Now we want to show that $\mathcal{C}$ links $(\frac{\lambda_{1}}{f_{0}},0)$ to $(\frac{\lambda_{1}}{f_{\infty}},\infty)$ in $\mathbb{R}\times X$. Let $(\lambda_{n},u_{n})\in \mathcal{C}$ satisfy $\lambda_{n}+\Vert u_{n}\Vert\rightarrow+\infty$ as $n\rightarrow+\infty$ and
    \begin{center}
    	$\begin{cases}
    			S_{k}(D^{2}u_{n})=\lambda_{n}^{k}f^{k}(-u_{n})\qquad &in \quad\Omega\\
    		u_{n}=0\qquad&on\quad\partial\Omega
    	\end{cases}$
    \end{center}
Since 0 is the only solution of \cref{eq:SK} for $\lambda=0$, so we have $\lambda_{n}>0$ for all $n\in \mathbb{N}$. We claim that there exists $M>0$ such that $\lambda_{n}\in \left(0,M\right]$. On the contrary, we suppose that $\lambda_{n}\rightarrow+\infty$. Since $f_{0},~f_{\infty}\in (0,+\infty)$, there exist a positive constant $\sigma$ such that 
    \begin{center}
    	$f(-u_{n})/\vert u_{n}\vert \ge \sigma~for ~any ~n\in \mathbb{N}$.
    \end{center}
So we have $\underset{n\rightarrow+\infty}{lim}\frac{\lambda_{n}^{k}f^{k}(-u_{n})}{\vert u_{n}\vert ^{k}}=+\infty$. By \cref{lem:Apriori} if necessary, for fixed n large, we always have $\Vert u_{n}\Vert \le C$ where $C>0$ is a constant and $\frac{\lambda_{n}^{k}f^{k}(-u_{n})}{\vert u_{n}\vert ^{k}}=\overline{\lambda}^{k}>\lambda^{k}_{1}$.
Let $v$ satisfies 
    \begin{center}
    	$\begin{cases}
    			S_{k}(D^{2}v)=\lambda_{1}^{k}\vert v\vert ^{k}\qquad &in \quad\Omega\\
    		v=0\qquad&on\quad\partial\Omega
    	\end{cases}$
    \end{center}
By scaling if necessary, for this n, we can assume $u_{n}(x)<v(x)$ for all $x\in \Omega$. Let $\delta^{*}>0$ be the maximal such that $u_{n}-\delta^{*}v\le 0$ in $\Omega$, i.e., there exists $x_{0}\in \Omega$ such that $u_{n}(x_{0})=v(x_{0})$. Let $\omega=\delta^{*}v$, $L_{\omega}=\Sigma_{i,j=1}^{n}F_{ij}(D^{2}\omega)D_{ij}$, then by \cite{caffarelli_dirichlet_1985} we have
    \begin{center}
    	$L_{k}(u_{n}-\delta^{*}v)\ge F_{k}(D^{2}u_{n})-F_{k}(D^{2}\omega)
    	=\lambda_{n}f(-u_{n})-(\lambda_{1}\delta^{*}\vert v\vert)
    	\ge (\overline{\lambda}\vert u_{n}\vert)-(\lambda_{1}\delta^{*}\vert v\vert)\ge 0$
    \end{center}
since $\overline{\lambda}\ge\lambda_{1}$ and $0<\sigma^{*}\vert v\vert\le \vert u_{n}\vert$ for all $x\in \Omega$. That implies, by strong maximum principle, $u_{n}(x)=\delta^{*}v$ for all $x\in \Omega$. Therefore, $S_{k}(D^{2}u_{n})=S_{k}(D^{2}\omega)$, or
    \begin{center}
    	$\lambda_{1}^{k}\vert u_{n}\vert ^{k}=\lambda_{1}^{k}\vert \delta^{*}v\vert ^{k}=\lambda_{n}^{k}f^{k}(-u_{n})> \lambda_{1}^{k}\vert u_{n}\vert^{k}$
    \end{center}
That is impossible, so only $\Vert u_{n}\Vert\rightarrow+\infty$ as $n\rightarrow+\infty$.

Let $\eta:\mathbb{R}^{+}\rightarrow\mathbb{R}^{+}$ be such that 
    \begin{center}
    $	f^{k}(s)=f_{\infty}^{k}s^{k}+\eta(s)$
    \end{center}
with $\underset{s\rightarrow+\infty}{lim}\eta(s)/s^{k}=0$. For $\Vert u_{n}\Vert \rightarrow+\infty$, we consider the following equation
    \begin{center}
    	$\begin{cases}
    	S_{k}(D^{2}u_{n})=\lambda_{n}^{k}(f_{\infty}^{k}\vert u_{n}\vert ^{k}+\eta (-u_{n}))\qquad &in \quad\Omega\\
    	u_{n}=0\qquad&on\quad\partial\Omega
    	\end{cases}$.
    \end{center}
Let $v_{n}=\frac{u_{n}}{\Vert u_{n}\Vert}$. Then $\Vert v_{n}\Vert=1$ and is the $k$-admissible solution to
    \begin{center}
    	$\begin{cases}
    		S_{k}(D^{2}w)=\lambda_{n}^{k}(f_{\infty}^{k}\vert w\vert ^{k}+\frac{\eta(-u_{n})}{	\Vert u_{n}\Vert^{k}})\qquad &in \quad\Omega\\
    	w=0\qquad&on\quad\partial\Omega
    	\end{cases}$
    \end{center}
Let
    \begin{center}
    	$\widetilde{\eta}(t)=\underset{0<\vert s\vert\le t}{max}\vert \eta(s)\vert$,
    \end{center}
then $\widetilde{\eta}(t)$ is non-decreasing with respect to $t$. Define
    \begin{center}
    	$\overline{\eta }(t)=\underset{t/2\le \vert s\vert \le t}{max}\vert \eta(s)\vert$.
    \end{center}
Then we have that 
    \begin{center}
    	$\underset{t \rightarrow+\infty}{lim}\frac{\overline{\eta }(t)}{\vert t\vert ^{k}}=0$ and $\widetilde{\eta}(t)\le \widetilde{\eta}(\frac{t}{2})+\overline{\eta }(t)$.
    \end{center}
It follows that 
    \begin{center}
    	$\underset{t \rightarrow+\infty}{limsup}~\frac{\widetilde{\eta }(t)}{\vert t\vert ^{k}}\le \underset{t\rightarrow+\infty}{limsup}~\frac{\widetilde{\eta }(\frac{t}{2})}{\vert t\vert ^{k}}=\underset{t/2\rightarrow+\infty}{limsup}~\frac{\widetilde{\eta }(\frac{t}{2})}{2^{k}\vert \frac{t}{2}\vert ^{k}}$.
    \end{center}
So we have that 
    \begin{equation}
    	\underset{t \rightarrow+\infty}{lim}\frac{\widetilde{\eta }(t)}{\vert t\vert ^{k}}=0.
    \end{equation}
It follows that
    \begin{equation}
    	\frac{\eta (-u)}{\Vert u\Vert ^{k}}\le 	\frac{\vert \eta (- u )\vert }{\Vert u\Vert ^{k}}\le 	\frac{\widetilde{\eta} (\Vert u\Vert )}{\Vert u\Vert ^{k}}\rightarrow0~as~\Vert u\Vert \rightarrow+\infty.
    \end{equation}
Since $v_{n}$ is bounded in X, after taking a sub-sequence if necessary, we have that $v_{n}\rightarrow v$ in $C^{0}(\overline{\Omega})$ as $n\rightarrow+\infty$ and $v$ solve the equation 
    \begin{center}
    	$\begin{cases}
    		S_{k}(D^{2}v)=\overline{\lambda}^{k}f_{\infty}^{k}\vert v\vert ^{k}\qquad &in \quad\Omega\\
    	v=0\qquad&on\quad\partial\Omega
    	\end{cases}$
    \end{center}
where $\overline{\lambda}=\underset{n\rightarrow}{lim}~\lambda_{n}$, then by Theorem 3.4 in \cite{jacobsen_global_1999}, we have $\overline{\lambda}=\frac{\lambda_{1}}{f_{\infty}}$. Therefore, $\mathcal{C}$ links $(\lambda_{1}/f_{0},0)$ to $(\lambda_{1}/f_{\infty},\infty)$. 

So \cref{eq:SK} has at least one admissible solution for every $\lambda\in (min\{\frac{\lambda_{1}}{f_{0}},\frac{\lambda_{1}}{f_{\infty}}\}$,$max\{\frac{\lambda_{1}}{f_{0}},\frac{\lambda_{1}}{f_{\infty}}\})$.
    
(3): By (1), there exists an unbounded continuum $\mathcal{C}$ of the set of $k$-admissible solution of \cref{eq:SK} emanating from $(\frac{\lambda_{1}}{f_{0}},0)$ such that $\mathcal{C}\subset(\mathbb{R}\times \Phi^{k}_{0})\cup \left\{(\frac{\lambda_{1}}{f_{0}},0)\right\}$. Then argument as in (1), we can prove that $\mathcal{C}$ is bounded in the direction of $\lambda$. So $\mathcal{C}$ is unbounded in the direction of $X$.
	
We claim that $0$ is the only blow-up points.
    
On the contrary, if there exists $\lambda_{*}>0$ such that $(\lambda_{n},u_{n})\in \mathcal{C}$ satisfying $\lambda_{n}\rightarrow\lambda_{*}$, $\Vert u_{n}\Vert \rightarrow+\infty$ as $n\rightarrow+\infty$. In view of $f_{\infty}=+\infty$, so for $n$ large enough, we have$\frac{\lambda_{n}^{k}f^{k}(-u_{n})}{\vert u_{n}\vert ^{k}}>\overline{\lambda}^{k}\ge \lambda_{1}^{k}$ for some constant $\overline{\lambda}\ge \lambda_{1}$, where $\lambda_{1}$ is the first eigenvalue of the Hessian equation. Let $v$ satisfies 
	\begin{center}
		$\begin{cases}
			S_{k}(D^{2}v)=\lambda_{1}^{k}\vert v\vert ^{k}\qquad &in \quad\Omega\\
			v=0\qquad&on\quad\partial\Omega
		\end{cases}$
	\end{center}
By scaling if necessary, for $n$ large enough, we can assume $u_{n}(x)<v(x)$ for all $x\in \Omega$. Let $\delta^{*}>0$ be the maximal such that $u_{n}-\delta^{*}v\le 0$ in $\Omega$, i.e., there exists $x_{0}\in \Omega$ such that $u_{n}(x_{0})=v(x_{0})$. Let $\omega=\delta^{*}v$, $L_{\omega}=\Sigma_{i,j=1}^{n}F_{ij}(D^{2}\omega)D_{ij}$, then by\cite{caffarelli_dirichlet_1985}
	\begin{center}
		$L_{w}(u_{n}-\delta^{*}v)\ge F_{k}(D^{2}u_{n})-F_{k}(D^{2}\omega)
		=\lambda_{n}f(-u_{n})-(\lambda_{1}\delta^{*}\vert v\vert)
		\ge (\overline{\lambda}\vert u_{n}\vert)-(\lambda_{1}\delta^{*}\vert v\vert)\ge 0$
	\end{center}
since $\overline{\lambda}\ge\lambda_{1}$ and $0<\sigma^{*}\vert v\vert\le \vert u_{n}\vert$ for all $x\in \Omega$. That implies, by strong maximum principle, $u_{n}(x)=\delta^{*}v$ for all $x\in \Omega$. Therefore, $S_{k}(D^{2}u_{n})=S_{k}(D^{2}\omega)$, or
	\begin{center}
		$\lambda_{1}^{k}\vert u_{n}\vert ^{k}=\lambda_{1}^{k}\vert \delta^{*}v\vert ^{k}=\lambda_{n}^{k}f^{k}(-u_{n})> \lambda_{1}^{k}\vert u_{n}\vert^{k}$
	\end{center}
which is impossible. Therefore, $\mathcal{C}$ links $(\frac{\lambda_{1}}{f_{0}},0)$ to $(0,\infty)$. Thus, \cref{eq:SK} has at least one admissible solution for every $\lambda\in (0,\frac{\lambda_{1}}{f_{0}})$. 
 
(5): Define
	\begin{center}
		$f^{k}_{n}(s)=	
		\begin{cases}
			(\frac{1}{n}s)^{k}+g^{n}(s) \qquad&s\in [0,\frac{1}{n}]\\
			f^{k}(s)\qquad&s\in \left[\frac{2}{n},+\infty\right)
		\end{cases}$
	\end{center}
where $\underset{s\rightarrow0}{lim}\frac{g^{n}(s)}{s^{k}}=0$. We have $\underset{n\rightarrow+\infty}{lim}f_{n}(s)=f(s)$ and $f_{n,0}=\frac{1}{n}$, $f_{n,\infty}=0$. Then, consider
	\begin{equation}
		\begin{cases}
			S_{k}(D^{2}u)=\lambda^{k}f^{k}_{n}(-u)\qquad&in\quad\Omega\\
			u=0\qquad&on\quad\partial\Omega
		\end{cases}
	\end{equation}
By (3), we get $\mathcal{C}_{n}$ of admissible solutions of the above problem emanating from $(n\lambda_{1},0)$ and links to $(0,\infty)$. Taking $z^{*}=(+\infty,0)$, we have that $z^{*}\in \underset{n\rightarrow+\infty}{lim}\mathcal{C}^{n}$. The compactness of $T_{k}$ in \cite{jacobsen_global_1999} implies that $(\bigcup_{n=1}^{+\infty}\mathcal{C}^{n})\cap B_{R}$ is precompact, where $B_{R}=\left\{z\in \mathbb{R}\times X:\Vert z\Vert_{\mathbb{R}\times X}<R\right\}$ for any $R>0$. Using Lemma 2.5 of \cite{dai_two_2016}, we obtain that $\mathcal{C}:=\underset{n\rightarrow+\infty}{limsup}~\mathcal{C}^{n}$ is an unbounded connected set such that $z^{*}\in \mathcal{C}$. Thus $\mathcal{C}$ links $(+\infty,0)$ to $(0,\infty)$. By the compactness of $T_{k}$, we immediately get $\mathcal{C}$ is the set of $k$-admissible solutions of \cref{eq:SK}, which hence has at least one admissible solution for every $\lambda\in (0,+\infty)$.

(6): Let
	\begin{center}
		$f^{k}_{n}(s)=	
		\begin{cases}
			(\frac{1}{n}s)^{k}+g^{n}(s) \qquad&s\in [0,\frac{1}{n}]\\
			f^{k}(s)\qquad&s\in \left[\frac{2}{n},+\infty\right)
		\end{cases}$
	\end{center}
where $\underset{s\rightarrow0}{lim}\frac{g^{n}(s)}{s^{k}}=0$. We have $\underset{n\rightarrow+\infty}{lim}f_{n}(s)=f(s)$ and $f_{n,0}=\frac{1}{n}$, $f_{n,\infty}=f_{\infty}$. Then, consider
	\begin{equation}
		\begin{cases}
			S_{k}(D^{2}u)=\lambda^{k}f^{k}_{n}(-u)\qquad&in\quad\Omega\\
			u=0\qquad&on\quad\partial\Omega
		\end{cases}
	\end{equation}
By (1), we get $\mathcal{C}_{n}$ of admissible solutions of the above problem emanating from $(n\lambda_{1},0)$ and links to $(\frac{\lambda_{1}}{f_{\infty}},\infty)$. Then argument as in (5), we would get the admissible set $\mathcal{C}$ of \cref{eq:SK} satisfies $(+\infty,0)\in \mathcal{C}$. Then similar to the proof of (1), we can get $\mathcal{C}$ links $(+\infty,0)$ to $(\frac{\lambda_1}{f_{\infty}},\infty)$ which hence has at least one $k$-admissible solution for every $\lambda\in (0,\frac{\lambda_{1}}{f_{\infty}})$.
\end{proof}

\subsection{Proof of \cref{thm:SK2,thm:SK3}}

We now deal with the extreme case of $f_{0}=f_{\infty}=+\infty$ inspired by \cite{luo_global_2020}.

\newenvironment{prf1}{{\noindent\bf Proof of \cref{thm:SK2}.}}{\hfill $\square$\par}
\begin{prf1}
In view of $f_{0}=+\infty$, from the argument of \cite{dai_global_2018} Theorem 1.1(b), we have that there is an unbounded continuum $\mathcal{C}$ of admissible solution of \cref{eq:SK} and $(0,0)\in \mathcal{C}$. Since $f_{0}=+\infty,~f_{\infty}=+\infty$ and the sign condition $f(s)>0$ for any $s>0$. If $\mathcal{C}$ is unbounded in the direction of $\lambda$, then argument as (1), we can get contradiction. It follows that $\mathcal{C}$ is unbounded in the direction of X.

We claim that $(0,0)$ is the unique blow up point of $\mathcal{C}$. Otherwise,there must exist a blow up point $(\widetilde{\lambda},0)$ with $\widetilde{\lambda}\in (0,+\infty )$ of $\mathcal{C}$. And then there exist$\left\{(\lambda_{n},u_{n})\right\}\in \mathcal{C}$ such that $\underset{n\rightarrow+\infty}{lim}\lambda_{n}=\widetilde{\lambda}$ and $\underset{n\rightarrow+\infty}{lim}\Vert u_{n}\Vert =+\infty$. Since $f_{\infty}=+\infty$, we have for any $M>0$, for $n$ large, $\frac{f^{k}(\Vert u_{n}\Vert)}{\Vert u_{n}\Vert ^{k}}\ge M$ which means $\underset{n\rightarrow+\infty}{lim}\lambda_{n}^{k}\frac{f^{k}(\Vert u_{n}\Vert)}{\Vert u_{n}\Vert ^{k}}=+\infty$. But by \cref{lem:Apriori}, we have that $\Vert u_{n}\Vert $ is bounded for any n which contradict with $\underset{n\rightarrow+\infty}{lim}\Vert u_{n}\Vert =+\infty$. Therefore, $(0,0)$ is the unique blow-up point. 

Let $\lambda^{*}=sup_{(\lambda,u)\in \mathcal{C}}\left\{\lambda\right\}$. It is clearly that $\lambda^{*}\in (0,+\infty)$ so \cref{eq:SK} has at least two nontrivial admissible solutions for $\lambda>0$ small enough. Obviously, there exists at least one nontrivial admissible solution at $\lambda=\lambda^{*}$. We divide the rest of proof into two steps.

{\bf Step 1}. For $\lambda>\lambda^{*}$, \cref{eq:SK} has no solution.

Assume that there exists an $k$-admissible solution $\overline{u}$ of \cref{eq:SK} for some $\overline{\lambda}>\lambda^{*}$. By \cref{lem:Apriori}, we have $\Vert \overline{u}\Vert_{0}\le C$ for some positive constant $C$ that only depend on $f,k~and ~\Omega$. Let $\mathcal{C}'=\left\{(\lambda,u)\in \mathcal{C}:u\ge\overline{u}\right\}$.

We $claim~\mathcal{C}'=\mathcal{C}$.

It is sufficient to show that $\mathcal{C}'$ is not empty, open and closed relative to $\mathcal{C}$. It is clearly that $\mathcal{C}$ is nonempty since $(0,0)\in \mathcal{C}$. The closeness of $\mathcal{C}'$ is obvious by the closeness of $\mathcal{C}$ and the definition of $\mathcal{C}'$. So we only need to show that $\mathcal{C}'$ is open relative to $\mathcal{C}$. For any $(\lambda_{0},u_{0})\in \mathcal{C}'$,we have that $u_{0}>\overline{u}~in ~\Omega$. For if not, there exists $x_{0}\in \Omega$ such that $u_{0}(x_{0})=\overline{u}(x_{0})$. Thus $x_{0}$ is a locally minimal point of $u_{0}-\overline{u}$, then $D_{ij}[u_{0}-\overline{u}](x_{0})$ is positive semi-definite. Since $u_{0},~\overline{u}\in \Phi^{k}(\Omega)$, by \cite{bhattacharya_maximum_2021}, we have
    \begin{center}
    	$S_{k}(D^{2}u_{0}(x_{0}))\ge S_{k}(D^{2}\overline{u}(x_{0}))$, i.e. $\lambda_{0}^{k}f^{k}(-u_{0}(x_{0}))\ge \overline{\lambda}^{k}f^{k}(-\overline{u}(x_{0}))$.
    \end{center}
But since $u_{0}(x_{0})=\overline{u}(x_{0})$ and $\overline{\lambda}>\lambda^{*}\ge \lambda_{0}$, The conclusion is impossible. Therefore, $u_{0}>\overline{u}~in ~\Omega$.

For any $\varepsilon>0$, we can choose a open neighborhood $B\subset \mathbb{R}\times X$ of$(\lambda_{0},u_{0})$ satisfying $\Vert u-u_{0}\Vert <\varepsilon$. Then for any $(\lambda,u)\in B\cap \mathcal{C}$, we have $u\ge\overline{u} $ in $\Omega_{\varepsilon}=\left\{x\in \Omega:u_{0}(x)\ge \overline{u}(x)+\varepsilon\right\}$.

We claim that there exists $\varepsilon$ small enough such that $u\ge\overline{u} $~in $\Omega^{'}:=\Omega\backslash\Omega_{\varepsilon}$.
In fact, let $L_{u}=\Sigma_{i,j=1}^{n}F_{ij}(u)D_{ij}$, then $L_{u}$ is elliptic linear operator as $u\in \Phi^{k}(\Omega)$. By \cite{caffarelli_dirichlet_1985},
    \begin{center}
    	$L_{u}(v-u)=\Sigma_{i,j=1}^{n}F_{ij}(u)D_{ij}(v-u)\ge F_{k}(v)-F_{k}(u)$
    \end{center}
for $u,~v\in \Phi^{k}(\Omega)$. Thus, by linearity, $L_{u}(u-v)\le F_{k}(u)-F_{k}(v)$. Let $v=\overline{u}$, $w(x)=u(x)-\overline{u}(x)$ then we have
    \begin{center}
    	$\begin{cases}
    	L_{u}(u-v)+h(x)w(x)\le 0\qquad&in ~\Omega'\\
    	w(x)\ge 0\qquad&in ~\partial \Omega'
    	\end{cases}$
    \end{center}
where $h(x)=\frac{\overline{\lambda}f(-\overline{u}(x))-\lambda f(-u(x))}{u(x)-\overline{u}(x)}$ is locally bounded. By theorem 2.32 of \cite{han_elliptic_2011}, $w(x)\ge 0$ in $\Omega'$ for $\varepsilon$ small enough. So $\mathcal{C}'$ is open relative to $\mathcal{C}$. Therefore, we get $\mathcal{C}'=\mathcal{C}$.

Since $\mathcal{C}'=\mathcal{C}$ is unbounded in the direction of X, there exists $(\lambda_{n},u_{n})\in \mathcal{C}'$ such that $\Vert u_{n}\Vert_{0}\rightarrow+\infty$ as $n\rightarrow+\infty$. So we have $\Vert \overline{u}\Vert_{0}=+\infty$, which is a contradiction.

{\bf Step 2}. For $\lambda<\lambda_{*}$, \cref{eq:SK} has at least two admissible solutions.

Let $\widetilde{\lambda}=sup\left\{\lambda>0:\exists u_{1}\not\equiv u_{2}~s.t.~ (\lambda,u_{1}),(\lambda,u_{2})\in \mathcal{C}\right\}$. Suppose on the contrary, that $\widetilde{\lambda}<\lambda^{*}$. Take $\hat{\lambda}\in \left[\widetilde{\lambda},\lambda^{*}\right)$ such that there exists a unique $\hat{u}$ such that $(\hat{\lambda},\hat{u})\in \mathcal{C}$. Define
    \begin{center}
    	$\mathcal{C}''=\left\{(\lambda,u)\in \mathcal{C}:\lambda\le \hat{\lambda}\right\}$.
    \end{center}
Clearly, $\mathcal{C}''$ is connected and closed. Then let $\overline{\mathcal{C}}=\left\{(\lambda,u)\in \mathcal{C}'':u\ge u''\right\}$, where $(\lambda'',u'')$ is a solution of \cref{eq:SK} and $\lambda''>\hat{\lambda}$. So we can repeat the argument of Step 1 to show that \cref{eq:SK} have no solution for $\lambda>\hat{\lambda}$, which is a contradiction.
\end{prf1}

\vskip 0.2in

We now deal with another extreme case of $f_{0}=f_{\infty}=0$.

\newenvironment{prf2}{{\noindent\bf Proof of \cref{thm:SK3}.}}{\hfill $\square$\par}

\begin{prf2}
Let 
\begin{center}
    $f^{k}_{n}(s)=	
    \begin{cases}
    	(\frac{1}{n}s)^{k}+g^{n}(s) \qquad&s\in [0,\frac{1}{n}]\\
    	f^{k}(s)\qquad&s\in \left[\frac{2}{n},+\infty\right)
    \end{cases}$
\end{center}
where $\underset{s\rightarrow0}{lim}\frac{g^{n}(s)}{s^{k}}=0$. We have $\underset{n\rightarrow+\infty}{lim}f_{n}(s)=f(s)$ and $f_{n,0}=\frac{1}{n}$, $f_{n,\infty}=0$. By Theorem 1.1(a) in \cite{dai_global_2018}, we get that continuum $\mathcal{C}_{n}$ of admissible solutions emanating $(n\lambda_{1},0)$ and is unbounded in the $\lambda$ direction. We claim that if $\left\{\lambda_{n},u_{n}\right\} \in \mathcal{C}_{n}$ satisfy $\lambda_{n}\rightarrow+\infty$ and $u_{n}\nrightarrow0$, then $\Vert u_{n}\Vert \rightarrow+\infty$. For if not, suppose that $0<\Vert u_{n}\Vert\le M$ for some constant $M>0$. Then, we have $\frac{f(-u_{n})}{\vert u_{n}\vert^{k} }\ge \delta$ for some $\delta >0$. Furthermore, for $n$ sufficient large, we have $\lambda_{n}^{k}\frac{f^{k}(-u_{n})}{\vert u_{n}\vert^{k} }>\lambda_{1}^{k}$. From the argument as (1), we will get contradiction. Hence, $\mathcal{C}_{n}$ link $(n\lambda_{1},0)$ to $(+\infty,\infty )$. Now, taking $z^{1}_{n}=(n\lambda_{1},0)$, $z^{2}_{n}=(+\infty,+\infty)$. The compactness of $T_{f}$ in \cite{jacobsen_global_1999} implies that $(\bigcup_{n=1}^{+\infty} \mathcal{C}_{n})\cap B_{R}$ is precompact. So by Lemma 4.11 in \cite{dai_eigenvalue_2015} there exists unbounded component $\mathcal{C}$ such that $(+\infty,0)\in \mathcal{C}$ and $(+\infty,\infty)\in  \mathcal{C}$.

Let $\lambda_{*}=\underset{(\lambda,u)\in  \mathcal{C}}{inf}\left\{\lambda\right\}$. Obviously, there exists at least one nontrivial admissible solution at $\lambda=\lambda_{*}$. We divide the rest of the proof into two steps.

{\bf Step 1}. When $\lambda<\lambda_{*}$ \cref{eq:SK} has no nontrivial solution.

Assume there exists a solution $\underline{u}$ of \cref{eq:SK} for some $\underline{\lambda}<\lambda_{*}$. By Lemma 6.1 of \cite{chou_variational_2001} we  have $\Vert\underline{u}\Vert_{0}\le C$ for some positive constant C. Let $\mathcal{C}'=\left\{(\lambda,u)\in \mathcal{C}:u\le \underline{u}\right\}$.

We $claim~\mathcal{C}'=\mathcal{C}$.

It is sufficient  to show that $\mathcal{C}'$ is not empty, open and closed relative to $\mathcal{C}$. It is clearly that $\mathcal{C}'$ is nonempty since $(+\infty,\infty)\in \mathcal{C}'$ and $f$ is coercive. In fact, for $(\lambda_{n},u_{n})\in \mathcal{C}$ satisfy $\Vert u_{n}\Vert \rightarrow+\infty$ and $\lambda_{n}\rightarrow+\infty$ as $n\rightarrow+\infty$. Then for $n$ large enough, since $f$ is coercive, we always have $\lambda_{n}^{k}f^{K}(-u_{n})\ge \underline{\lambda}^{k}f^{K}(-\underline{u})$ in $\Omega$. Thus, by comparison principle, we have $u_{n}\le \underline{u}$ for $n$ large enough. The closeness of $\mathcal{C}'$ is obvious by the closeness of $\mathcal{C}$ and the definition of $\mathcal{C}'$. So we only need to show that $\mathcal{C}'$ is open relative to $\mathcal{C}$. For any $(\lambda_{0},u_{0})\in \mathcal{C}'$, we have that $u_{0}<\underline{u}~in ~\Omega$. For if not, There exists $x_{0}\in \Omega$ such that $u_{0}(x_{0})=\underline{u}(x_{0})$, then we have $x_{0}$ is a locally minimum of $\underline{u}-u_{0}$, then $D_{ij}[\underline{u}-u_{0}](x_{0})$ is positive semi-definite. Since $u_{0},~\underline{u}\in \Phi^{k}(\Omega)$, by \cite{bhattacharya_maximum_2021}, we have
    \begin{center}
    	$S_{k}(D^{2}\underline{u}(x_{0}))\ge S_{k}(D^{2}u_{0}(x_{0}))$, i.e. $\underline{\lambda}^{k}f^{k}(-\underline{u}(x_{0}))\ge \lambda_{0}^{k}f^{k}(-u_{0}(x_{0}))$.
    \end{center}
But since $u_{0}(x_{0})=\underline{u}(x_{0})$ and $\underline{\lambda}<\lambda^{*}\le \lambda_{0}$, The conclusion is impossible. Therefore, $u_{0}<\underline{u}~in ~\Omega$.

For any $\varepsilon>0$, we choose an open neighborhood $B\subset \mathbb{R}\times X$ of $(\lambda_{0},u_{0})$ such that $\Vert u-u_{0}\Vert <\varepsilon$. For any $(\lambda,u)\in B\cap \mathcal{C}$, we have $u\le \underline{u} $ in $\Omega_{\varepsilon}=\left\{x\in \Omega:u_{0}(x)\le \overline{u}(x)-\varepsilon\right\}$. 

We claim that there exists $\varepsilon$ small enough such that $u\le \underline{u} $ in $\Omega':=\Omega\backslash\Omega_{\varepsilon}$.

In fact, let $L_{\omega}=\Sigma_{i,j=1}^{n}F_{ij}(\omega)D_{ij}$, then $L_{\omega}$ is elliptic linear operator as $\omega\in \Phi^{k}(\Omega)$. By \cite{caffarelli_dirichlet_1985},
    \begin{center}
    	$L_{\omega}(v-\omega)=\Sigma_{i,j=1}^{n}F_{ij}(\omega)D_{ij}(v-\omega)\ge F_{k}(v)-F_{k}(\omega)$
    \end{center}
for $\omega,~v\in \Phi^{k}(\Omega)$. Thus, by linearity, $L_{\omega}(\omega-v)\le F_{k}(\omega)-F_{k}(v)$. Let $\omega=\underline{u}$, $v=u$, $w(x)=\underline{u}-u(x)$ then we have
    \begin{center}
    	$\begin{cases}
    		L_{\underline{u}}(\underline{u}-u)+h(x)w(x)\le 0\qquad&in ~\Omega'\\
    		w(x)\ge 0\qquad&in ~\partial\Omega'
    	\end{cases}$
    \end{center}
where $h(x)=\frac{\lambda f(-u(x))-\underline{\lambda}f(-\underline{u}(x))}{\underline{u}-u(x)}$ is locally bounded. By the maximum principle on narrow domains (see \cite{han_elliptic_2011}), $w(x)\ge 0$ in $\Omega'$ for $\varepsilon$ small enough. So $\mathcal{C}'$ is open relative to $\mathcal{C}$. Therefore, we get $\mathcal{C}'=\mathcal{C}$.

So for all $u\in \mathcal{C}$, we have $u(x)\le  \underline{u}(x)$. But $(+\infty,0)\in \mathcal{C}$ and $\underline{u}\le 0$, so $\underline{u}\equiv0$ which is in contradiction with the choose of nontrivial solution $\underline{u}$. So when $\lambda<\lambda_{*}$, \cref{eq:SK} has no solution.

{\bf Step 2 }.For $\lambda>\lambda_{*}$, \cref{eq:SK} has at least two nontrivial admissible solutions.

Let $\widetilde{\lambda}=inf\left\{\lambda>0:\exists u_{1}\not\equiv u_{2}~s.t.~(\lambda,u_{1}),(\lambda,u_{2})\in \mathcal{C}\right\}$. Suppose on the contrary, that $\widetilde{\lambda}>\lambda_{*}$. Take $\lambda'\in \left(\lambda_{*}, \widetilde{\lambda}\right]$such that there exists a unique $u$ such that $(\lambda',u)\in \mathcal{C}$. Define
    \begin{center}
    	$\mathcal{C}''=\left\{(\lambda,u)\in \mathcal{C}:\lambda\ge \lambda'\right\}$.
    \end{center}
Clearly, $\mathcal{C}''$ is connected and closed. Then let $\overline{\mathcal{C}}=\left\{(\lambda,u)\in \mathcal{C}'':u\le  u''\right\}$, where $(\lambda'',u'')$ is a nontrivial solution of \cref{eq:SK} and $\lambda''<\lambda'$. So we can repeat the argument of Step 1 to show the \cref{eq:SK} have no solution for $\lambda<\lambda'$, which is a contradiction.
\end{prf2}
     
\begin{remark} 
\label{rmk1}
In fact, If we suppose that $f$ is coercive then for $\lambda>0$ small, we can always construct a sub-solution of \cref{eq:SK}. Notice that 0 is a super-solution of \cref{eq:SK}, if we want to prove that there exists $\lambda_{*}$ such that for all $\lambda<\lambda_{*}$, then \cref{eq:SK} has no nontrivial solution, which means we can only get trivial solution by super-sub solution method.
\end{remark}
    
For generally semi-linear elliptic, it is easy to get counterexample. For instance, consider the existence of positive solution of the following equation
    \begin{center}
    	$\begin{cases}
    		-\Delta u=\lambda f(u)\qquad&in\quad\Omega\\
    		u=0 \qquad&on\quad \partial\Omega
    	\end{cases}$
    \end{center}
where $\lambda>0$ is parameter and let $F(u)=\int_{0}^{u}f(s)~dx$. Because of $f_{0}=\underset{s\rightarrow0}{lim}\frac{f(s)}{s}=0$ and $f_{\infty}=\underset{s\rightarrow\infty}{lim}\frac{f(s)}{s}=0$,then there exists $M>0$ such that $ F(u) \le M u^{2}$ for all $u\ge 0$. Thus, there exists $C>0$ such that $0\le  \int_{\Omega}F(u)dx \le C\Vert u\Vert^{2}_{H^{1}_{0}(\Omega)}$. So we can study $J(u):=\int_{\Omega}F(u)dx$ by variational method with one constraint. Then let $\mathcal{S}=\left\{u\in H_{0}^{1}(\Omega):u\ge 0,~\Vert u\Vert_{H^{1}_{0}}(\Omega)\le1\right\}$ and consider $\underset{u\in \mathcal{S}}{sup} J(u)$, we will get a weak solution $(\lambda_{u},u)$ of above problem. Further more, by scaling if necessary, for some $f$ we can hold $\lambda$ small.
   
By \cref{rmk1} of semi-linear equation and \cref{thm:MA2} of Monge-Amp\`ere equation, we have the following conjecture.
\begin{conj}
    If $f_{0}=0$ and $f_{\infty}=0$, then there exists $k^{*}>1$ such that \cref{eq:SK} has no nontrivial solution for all $\lambda\in (0,\lambda_{*})$ and \cref{eq:SK} has at least two nontrivial solution for all $\lambda\in (\lambda_{*},+\infty)$ as $k\ge k^{*}$, where $\lambda_{*}>0$.
\end{conj}

\subsection{Monge-Amp\`ere equation}

We now discuss a special case when $k=N$, i.e. the so called Monge-Amp\`ere equation case. Due to the convexity of the solution in this case, we can have a better result without additional condition on $f$.

\begin{theorem}[\cref{thm:MA2}]
If $f_{0}=0$ and $f_{\infty}=0$, then there exists a positive constant $\lambda_{*}$ such that \cref{eq:MA} has at least two convex solutions for all $\lambda>\lambda_{*}$, exactly one convex solution for $\lambda=\lambda_{*}$ and no nontrivial solution for all $\lambda\in (0,\lambda_{*})$.   
\end{theorem}
    
\begin{proof} 
Let 
    \begin{center}
    $f^{n}(s)=	
    \begin{cases}
    	\frac{1}{n}	s+g^{n}(s) \qquad&s\in [0,\frac{1}{n}]\\
    	f(s)\qquad&s\in \left[\frac{2}{n},+\infty\right)
    \end{cases}$
    \end{center}
where $\underset{n\rightarrow0}{lim}\frac{g^{n}(s)}{s}=0$. We have $\underset{n\rightarrow+\infty}{lim}f^{n}(s)=f(s)$ and $f^{n}_{0}=\frac{1}{n}$, $f^{n}_{0}=0$. By Theorem 1.2(2) in \cite{luo_global_2020} we get that continuum  $\mathcal{C}_{n}$ joint $(n\lambda_{1},0)$ to $(+\infty,+\infty)$. Now, taking $Z^{1}_{n}=(n\lambda_{1},0)$, $Z^{2}_{n}=(+\infty,+\infty)$. The compactness of $T_{f}$ (see \cite{jacobsen_global_1999}) implies that $(\bigcup_{n=1}^{+\infty} \mathcal{C}_{n})\cap B_{R}$ is precompact, where $B_{R}=\left\{Z\in \mathbb{R}\times X:\Vert Z\Vert < R\right\}$. So by Lemma 4.11 in \cite{dai_eigenvalue_2015} there exists unbounded component $\mathcal{C}$ such that $(+\infty,0)\in \mathcal{C}$ and $(+\infty,\infty)\in  \mathcal{C}$.

Let $\lambda_{*}=\underset{(\lambda,u)\in  \mathcal{C}}{inf}\left\{\lambda\right\}$. Obviously, there exists at least one convex solution at $\lambda=\lambda_{*}$. We divide the rest of the proof into two steps.

{\bf Step 1}. When $\lambda<\lambda_{*}$ \cref{eq:SK} has no solution.

Assume there exists a solution $\underline{u}$ of \cref{eq:SK} for some $\underline{\lambda}<\lambda_{*}$. By Lemma 6.1 in \cite{chou_variational_2001} We have $\Vert\underline{u}\Vert_{0}\le C$ for some positive constant C. Let $\mathcal{C}'=\left\{(\lambda,u)\in \mathcal{C}:u\le \underline{u}\right\}$.

We $claim~\mathcal{C}'=\mathcal{C}$.
    
It is sufficient to show that $\mathcal{C}'$ is not empty, open and closed relative to $\mathcal{C}$. It is clearly that $\mathcal{C}'$ is nonempty since $(+\infty,-\infty)\in \mathcal{C}'$ and the solutions are convex. The closeness of $\mathcal{C}'$ is obvious by the closeness of $\mathcal{C}$ and the definition of $\mathcal{C}'$. So we only need to show that $\mathcal{C}'$ is open relative to $\mathcal{C}$. For any $(\lambda_{0},u_{0})\in \mathcal{C}'$, we have that $u_{0}<\underline{u}~in ~\Omega$. For if not, There exists $x_{0}\in \Omega$ such that $u_{0}(x_{0})=\underline{u}(x_{0})$,then we have $x_{0}$ is a locally minimum of $\underline{u}-u_{0}$, so $det~(D^{2}(\underline{u}-u_{0})(x_0))\ge 0$. Then, $det(D^{2}\underline{u}(x_{0}))-det(D^{2}u_{0}(x_{0}))\ge det(D^{2}(\underline{u}-u_{0}))\ge 0$ i.e. $\underline{\lambda}^{N}f^{N}(-\underline{u}(x_{0}))\ge \lambda_{0}^{N}f^{N}(-u_{0}(x_{0}))$. That is impossible since $\underline{\lambda}<\lambda_{*}\le \lambda_{0}$ and $u_{0}(x_{0})=\underline{u}(x_{0})$. 

By the definition of $\mathcal{C}'$, $\mathcal{C}\backslash\mathcal{C}'=\left\{u\in\mathcal{C}:\exists x\in \Omega ~s.t.~u(x)<\overline{u}\right\}$. If $\mathcal{C}\backslash\mathcal{C}'=\emptyset$,then our claim follows. If $\mathcal{C}\backslash\mathcal{C}'\ne \emptyset$, but if we can prove that $\mathcal{C}\backslash\mathcal{C}'$ is closed, it means that $\mathcal{C}$ is the union of two nonempty closed sets which is in contradiction with the connection of $\mathcal{C}$. So we have $\mathcal{C}\backslash\mathcal{C}'=\emptyset$, thus our claim follows.

To prove that $\mathcal{C}\backslash\mathcal{C}'$is a closed set. Argument on the contrary, assume that $\mathcal{C}\backslash\mathcal{C}'$is not closed, then there exists a sequence $\left\{u_{n}\right\}\subset\mathcal{C}\backslash\mathcal{C}'$, $x_{n}\in \Omega$ and $u'\in \mathcal{C}'$ such that $u_{n}\rightarrow u'$ as $n\rightarrow+\infty$ and $u_{n}(x_{n})>\underline{u}(x_{n})>u'(x_{n})$ for any $n\in \mathbb{N}$. It is obvious that $x_{n}\rightarrow x_{*}$ as $n\rightarrow+\infty$, where $x_{*}\in \partial\Omega$. But it is in contradiction with the convexity.

In fact, for $x_{n}$ which satisfies $u_{n}(x_{n})>\underline{u}(x_{n})$, by the convexity, there exists $\widetilde{x}\in \Omega$ which is in the neighborhood of $x_{n}$ such that $u_{n}(\widetilde{x})=\underline{u}(\widetilde{x})$. Since $x_{n}\rightarrow x_{*}$ as $n\rightarrow+\infty$, where $x_{*}\in \partial\Omega$. Then for $n$ large enough, without loss of generality, we can assume that $\widetilde{x}$ is the interior point which is closest to $\partial\Omega$ such that $u_{n}(\widetilde{x})=\underline{u}(\widetilde{x})$. By convexity, $u_{n}-\underline{u}$ not change sign in the convex combination of $\widetilde{x}$ and boundary point. Furthermore, there exists $\hat{x}\in \Omega$ which is in the neighborhood of $\widetilde{x}$ such that $u_{n}(\hat{x})<\underline{u}(\hat{x})$. Let $h(t)=tu_{n}(\hat{x})+(1-t)u_{n}(x_{n})$, where $t\in (0,1)$. Then ,there exists $\overline{x}=t_{1}\hat{x}+(1-t_{1})x_{n}$  such that $h(t_{1})=\underline{u}(\overline{x})$. We can always choose suitable $\hat{x}$and $x_{n}$ such that $u_{n}(\overline{x})\ge \underline{u}(\overline{x})$ for $n$ large enough. Thus, by the definition of $\widetilde{x}$, we have $u_{n}(\overline{x})> \underline{u}(\overline{x})$, i.e., $u_{n}(t_{1}\hat{x}+(1-t_{1})x_{n})>t_{1}u_{n}(\hat{x})+(1-t_{1})u_{n}(x_{n})$, which is in contradiction with the convexity of $u_{n}$.

Hence, $\mathcal{C}\backslash\mathcal{C}'$ is a closed set which means that $\mathcal{C}'=\mathcal{C}$.

So for all $u\in \mathcal{C}$, we have $u(x)\le  \underline{u}(x)$. But $(+\infty,0)\in \mathcal{C}$ and $\underline{u}\le 0$, so $\underline{u}\equiv0$ which is in contradiction with the choose of nontrivial solution $\underline{u}$. So when $\lambda<\lambda_{*}$, \cref{eq:MA} has no solution.
    
{\bf Step 2}. For $\lambda>\lambda_{*}$, \cref{eq:MA} has at least two convex solutions.

Let $\widetilde{\lambda}=inf\left\{\lambda>0:\exists u_{1}\not\equiv u_{2}~s.t.~(\lambda,u_{1}),(\lambda,u_{2})\in \mathcal{C}\right\}$. Suppose on the contrary, that $\widetilde{\lambda}>\lambda_{*}$. Take $\lambda'\in \left(\lambda_{*}, \widetilde{\lambda}\right]$ such that there exists a unique $u$ such that $(\lambda',u)\in \mathcal{C}$. Define
    \begin{center}
    	$\mathcal{C}''=\left\{(\lambda,u)\in \mathcal{C}:\lambda\ge \lambda'\right\}$.
    \end{center}
Clearly, $\mathcal{C}''$ is connected and closed. Then let $\overline{\mathcal{C}}=\left\{(\lambda,u)\in \mathcal{C}'':u\le  u''\right\}$, where $(\lambda'',u'')$ is a nontrivial solution of (1.1) and $\lambda''<\lambda'$. So we can repeat the argument of Step 1 to show the \cref{eq:MA} have no solution for $\lambda<\lambda'$, which is a contradiction.
\end{proof}

\section{Hessian Systems}
\label{sec:HS}

Next, we are going to consider the system case, we take 2-coupled case as example for convenient, all the following results are still holds to m-coupled $(m\geq2)$ case respectively.

We study the coupled $k$-Hessian systems
    \begin{equation}\label{eq:CSK1}
    	\begin{cases}
    		(S_{k}(D^{2}u))^{1/k}=\lambda g(-u,-v) \qquad& in\quad\Omega \\  
    		(S_{k}(D^{2}v))^{1/k}=\lambda h(-u,-v)\qquad& in\quad\Omega \\  
    		u=v=0                 \qquad& on\quad \partial\Omega
    	\end{cases}
    \end{equation}
where $\Omega$ is a $(k$-$1)$-convex domain in $\mathbb{R}^{N}$ and $\lambda\ge 0$.

\subsection{Eigenvalue Problems of Hessian systems}
In order to study \cref{eq:CSK1} by bifurcation we need the following eigenvalue problem.
    \begin{equation}\label{eq:CSK2}
    	\begin{cases}
    		S_{k}(D^{2}u)=\lambda(-v)^{\alpha} \qquad& in\quad\Omega \\  
    		S_{k}(D^{2}v)=\mu(-u)^{\beta}\qquad& in\quad\Omega \\  
    		u=v=0                 \qquad& on\quad \partial\Omega
    	\end{cases}
    \end{equation}

We firstly consider the eigenvalue problems of Hessian systems. Inspired by \cite{zhang_power-type_2015}, we have the following conclusion.
\begin{prop} 
\label{pro:SKEP} 
Suppose $\Omega\subset\mathbb{R}^{N}$ is a bounded $(k$-$1)$-convex domain. If $\alpha>0$ and $\beta>0$, $\alpha\beta=k^{2}$, the \cref{eq:CSK2} admits a nontrivial $k$-convex solution if and only if $\lambda\mu ^{\frac{\alpha}{k}}=C$, where $C=C(k,\alpha,\Omega)$.
\end{prop}
    
To prove \cref{pro:SKEP}, we need the following lemma which we can find in\cite{jacobsen_global_1999}.
    
\begin{lemma} \label{lem:cco}
Let E be real Banach space which contains a cone M. Let $A:E\rightarrow E$ be a completely continuous operator with $A|_{M}:M\rightarrow M$ homogeneous, monotonous and strong. Furthermore, assume that there exist nonzero elements $\omega$, $A(\omega)\in Im(A)\cap M$. Then there exists a constant $\lambda_{1}>0$ with the following properties:
\begin{enumerate}
    \item There exists $u\in M\backslash\left\{0\right\}$, with $u=\lambda_{1}A(u)$,
    \item If $v\in M\backslash\left\{0\right\}$ and $\lambda>0$ such that $v=\lambda A(v)$, then $\lambda=\lambda_{1}$.
\end{enumerate}
\end{lemma}

\begin{lemma}\label{lem:SKLp}
Let $\Omega$ be a $(k$-$1)$-convex domain, then the Dirichlet problem
    \begin{center}
    	$\begin{cases}
    		S_{k}(D^{2}u)=\psi(x)\qquad&in\quad\Omega \\
    		u=\phi  \qquad&on\quad\partial\Omega
    	\end{cases}$
    \end{center}
has a unique $k$-convex solution $u\in C(\overline{\Omega})$ for any $\psi \in L^{p}(\Omega)$ with $p\ge \frac{N}{2k}$, and $\phi\in C(\overline{\Omega})$.
\end{lemma}

\begin{lemma}\label{lem:SKcp}
Let $u,v\in \Phi^{k}(\Omega)\cap C(\overline{\Omega})$ satisfy
    \begin{center}
    	$\begin{cases}
    		\mu_{k}[u]\le \mu_{k}[v]\qquad&in\quad\Omega \\
    		u\le v \qquad&on\quad\partial\Omega
    	\end{cases}$
    \end{center}
on the $(k$-$1)$-convex domain $\Omega$. Then $u\le v$ in $\Omega$.
\end{lemma}

\begin{proof}
Define $A_{1}:X\rightarrow X$ with $A_{1}(u)=v$, where $v$ is the unique $k$-admissible solution of the equation
    \begin{equation}\label{eq:SKalpha}
    	\begin{cases}
    		S_{k}(D^{2}v)=\vert u\vert ^{\alpha}\qquad&in\quad\Omega\\
    		v=0  \qquad&on\quad\partial\Omega
    	\end{cases}
    \end{equation}
By \cref{lem:SKLp}, $A_{1}$ is well-defined. Similarly, we define $A_{2}:X\rightarrow X$ with $A_{2}(u)=v$, where $v$ is the unique $k$-admissible solution of the equation
    \begin{equation}\label{eq:SKbeta}
    	\begin{cases}
    		S_{k}(D^{2}v)=\vert u\vert ^{\beta}\qquad&in\quad\Omega\\
    		v=0  \qquad&on\quad\partial\Omega
    	\end{cases}
    \end{equation}
By \cref{lem:SKLp}, for the admissible solution of \cref{eq:SKalpha,eq:SKbeta} we can see that $A_{1}u\in C^{1}(\overline{\Omega})$, $A_{2}u\in C^{1}(\overline{\Omega})$. Furthermore, we define a composite operator $A=A_{1}A_{2}$. We need to verify $A$ satisfies the assumptions of \cref{lem:cco}. Firstly, $A_{1}$, $A_{2}$ (thus $A$) are completely continuous by \cite{jacobsen_global_1999}. $A$ is positive since $A(X)\subset P$. Then let $t>0$, we have that
    \begin{center}
    	$A_{1}(tu)=t^{\alpha/k}A_{1}u,\quad A_{2}(tu)=t^{\beta/k}A_{2}u$.
    \end{center}
    As $\alpha\beta=k^{2}$, we deduce 
    \begin{center}
    	$A(tu)=A_{1}A_{2}(tu)=A_{1}(t^{\beta/k}A_{2}u)=tA_{1}A_{2}(u)=tA(u)$,
    \end{center}
which implies that $A$ is homogeneous. Besides, it is easy to get that $A_{1},~A_{2}$ are monotone operator by \cref{lem:SKcp}, so is $A$. To see $A$ is strong, notice that if $u\in Im(A)\cap M\backslash\left\{0\right\}$, then there exists a $v\in X\backslash\left\{0\right\}$ such that $u=A_{1}(A_{2}v)$. Indeed, $A_{2}v$ is a nonzero $k$-convex function that is strictly negative in $\Omega$. By \cref{lem:SKLp}, we see that $A_{2}v\in C^{1}(\overline{\Omega}), ~A_{1}(A_{2}v)\in  C^{1}(\overline{\Omega})$. It follows that $u$ is $k$-convex thus sub-harmonic, so $u(x)<0$ for $x\in \Omega$. Consequently, $A$ is a strong operator. By \cref{lem:cco}, there exist $u_{*}\in\Phi^{k}_{0} \backslash\left\{0\right\}$ and $\lambda_{1}>0$ such that
    \begin{center}
    	$u_{*}=\lambda_{1}A(u_{*})$.
    \end{center}
If we define $v_{*}=A_{2}(u_{*})$, then $(u_{*},v_{*})$ must be a solution of the following system
    \begin{center}
    	$\begin{cases}
    		S_{k}(D^{2}\frac{u}{\lambda_{1}})=(-v)^{\alpha} \qquad& in\quad\Omega \\
    		S_{k}(D^{2}v)=(-u)^{\beta}\qquad& in\quad\Omega \\
    		u<0,~v<0\qquad& in\quad\Omega \\
    		u=v=0                 \qquad& on\quad \partial\Omega
    	\end{cases}$
    \end{center}
Furthermore, by the second conclusion of \cref{lem:cco}, if $u_{1}\in \Phi^{k}_{0}\backslash\left\{0\right\}$ and $\lambda>0$ satisfy $u_{1}=\lambda A(u_{1})$, then $\lambda=\lambda_{1}$. So the following system 
    \begin{equation}\label{eq:SKalphabeta}
    	\begin{cases}
    		S_{k}(D^{2}u)=\widetilde{\lambda}(-v)^{\alpha} \qquad& in\quad\Omega \\
    		S_{k}(D^{2}v)=(-u)^{\beta}\qquad& in\quad\Omega \\
    		u<0,~v<0\qquad& in\quad\Omega \\
    		u=v=0                 \qquad& on\quad \partial\Omega
    	\end{cases}
    \end{equation}
admits a solution if and only if $\widetilde{\lambda}=\lambda^{k}_{1}$. 

Next we show that \cref{eq:CSK2} has a nontrivial $k$-admissible solution if and only if $\lambda\mu ^{\alpha/k}=\lambda^{k}_{1}$ which implies the conclusion of \cref{pro:SKEP}. Indeed, if $(u,v)$ is a nontrivial $k$-asmissible solution of \cref{eq:CSK2} then from 
    \begin{center}
    	$S_{k}(D^{2}v)=\mu (-u)^{\beta}$,
    \end{center}
    we have 
    \begin{center}
    	$	S_{k}(D^{2}(\mu^{-\frac{1}{k}}v))=(-u)^{\beta}$.
    \end{center}
    Let $\widetilde{v}=\mu^{-\frac{1}{k}}v$, then $(-v)^{\alpha}=\mu^{\alpha/k}(-\widetilde{v})^{\alpha}$. It is easy to see that $(u,\widetilde{v})$ is a $k$-admissible solution of \cref{eq:SKalphabeta} if and only if $\lambda_{1}^{k}=\widetilde{\lambda}=\lambda\mu^{\alpha/k}$.

    On the other hand, if $\lambda_{1}^{k}=\widetilde{\lambda}=\lambda\mu^{\alpha/k}$, then  \cref{eq:SKalphabeta} have nontrivial $k$-admissible solution $(u,v)$. Denote $u_{*}=u$, $v_{*}=\mu^{\frac{1}{k}}v$, then $(u_{*},v_{*})$ is a $k$-admissible of \cref{eq:CSK2}.
\end{proof}

We can immediately get the following corollary for the eigenvalue problem of $k$-Hessian systems.
\begin{coro}\label{coro:SKSE} 
The eigenvalue problem
    \begin{equation}\label{eq:SKSE}
    	\begin{cases}
    		(S_{k}(D^{2}u))^{\frac1k}=\vert \lambda	v\vert  \qquad& in\quad\Omega \\
    		(S_{k}(D^{2}v))^{\frac1k}=\vert \lambda u\vert \qquad& in\quad\Omega \\
    		u=v=0 \qquad& on\quad \partial\Omega
    	\end{cases}
    \end{equation}
admits nontrivial $k$-convex solutions if and only if $\vert\lambda\vert=\lambda_{0}(\Omega)$.
\end{coro}

Let E be the Banach space $C(\overline{\Omega})\times C(\overline{\Omega}) $ with norm $\Vert (u,v)\Vert :=\Vert u\Vert_{\infty}+\Vert v\Vert_{\infty}$. Define a cone
    \begin{center}
    	$M:=\left\{(u,v)\in E:~u(x)\le 0,~v(x)\le 0, \forall x\in \Omega\right\}$,
    \end{center} 
then $M$ induced a partial order on $E$ via $(u_{1},v_{1})\preccurlyeq(u_{2},v_{2})\iff (u_{2}-u_{1},v_{2}-v_{1})\in M$. We firstly define the solution operator $B: E\rightarrow E$ associated to the following problem as $B(u,v)=(w,z)$, where $(w,z)$ is the solution pair of 
    \begin{equation}
    	\begin{cases}
    		S_{k}(D^{2}w)=\vert v\vert ^{k}\qquad& in\quad\Omega \\
    		S_{k}(D^{2}z)=\vert u\vert ^{k}\qquad& in\quad\Omega \\
    		w=z=0                 \qquad& on\quad \partial\Omega
    	\end{cases}
    \end{equation}
We can write $B=(B_{1},B_{2})$, where 
    \begin{equation}
    	\begin{aligned}
    		&B_{1}(u,v)=T_{k}(\vert  v\vert ^{k}) ,\\
    		&B_{2}(u,v)=T_{k}(\vert  u\vert ^{k}).\nonumber
    	\end{aligned}
    \end{equation}
By the complete continuity of $T_{k}$ (see \cite{jacobsen_global_1999}), it is easy to check $B_{1}$ and $B_{2}$ are both completely continuous. So $B$ is completely continuous.
    
Next we consider the following auxiliary problem via bifurcation,
    \begin{equation}\label{eq:SKaux}
    	\begin{cases}
    		S_{k}(D^{2}u)=\vert \lambda v\vert ^{k} +\vert \lambda \vert ^{k}f(u,v)\qquad& in\quad\Omega \\
    		S_{k}(D^{2}v)=\vert \lambda u\vert ^{k}+\vert \lambda \vert ^{k}g(u,v)\qquad& in\quad\Omega \\
    		u=v=0                 \qquad& on\quad \partial\Omega
    	\end{cases}.	
    \end{equation}
with an assumption 

${\bf (A_{1})}$ $f,g: \mathbb{R}\times \mathbb{R}\rightarrow\mathbb{R}$ are continuous, moreover, $\vert t\vert^{k}+f(s,t)>0$ and $\vert s\vert^{k}+g(s,t)>0$ for any $s,t\ne 0$.

We define the solution operator $H:\mathbb{R}^{+}\times E\rightarrow E$ associated to \cref{eq:SKaux} as $H(\lambda,(u,v))=(w,z)$, where $(w,z)$ is the solution pair of 
    \begin{equation}
    	\begin{cases}
    		S_{k}(D^{2}w)=\vert \lambda v\vert ^{k} +\vert \lambda \vert ^{k}f(u,v)\qquad& in\quad\Omega \\
    		S_{k}(D^{2}z)=\vert \lambda u\vert ^{k}+\vert \lambda \vert ^{k}g(u,v)\qquad& in\quad\Omega \\
    		w=z=0                 \qquad& on\quad \partial\Omega
    	\end{cases}.
    \end{equation}
We can write $H=(H_{1},H_{2})$, where 
    \begin{equation}
    	\begin{aligned}
    		&H_{1}(\lambda,(u,v))=T_{k}(\vert \lambda v\vert ^{k} +\vert \lambda \vert ^{k}f(u,v)), \\
    		&H_{2}(\lambda,(u,v))=T_{k}(\vert \lambda u\vert ^{k} +\vert \lambda \vert ^{k}g(u,v)).\nonumber
    	\end{aligned}
    \end{equation}
By $(A_{1})$ and the complete continuity of $T_{k}$ (see [Jacobsen]), it is easy to check $H_{1}$ and $H_{2}$ are both completely continuous. So $H$ is completely continuous. Define $h:\mathbb{R}^{+}\times E\rightarrow E,~h(\lambda,(u,v))=(u,v)-H(\lambda, (u,v))$, and consider the equation
    \begin{equation}\label{eq:h}
    	h(\lambda,(u,v))=0.
    \end{equation}
We see that $\lambda,~u~,v$ satisfy \cref{eq:SKaux} if and only if $(\lambda,(u,v))$ is a solution of \eqref{eq:h}.

The following lemma comes from \cite{jacobsen_global_1999}.
\begin{lemma}\label{lem:HB} $H$ is a completely continuous, positive operator. $B: E\rightarrow E$ is a completely continuous, monotone operator, moreover, it is homogeneous.
\end{lemma}

Note that \cref{eq:SKaux} can be seen as a perturbation of the \cref{eq:SKSE}. For our purpose in this section, the perturbation terms also need to satisfy:
    
${\bf (A_{2})}$ $f(s,t)=o((\vert s\vert +\vert t\vert )^{k}),~g(s,t)=o((\vert s\vert +\vert t\vert )^{k}), ~as~ \vert s\vert +\vert t\vert \rightarrow0.$

\begin{lemma}\label{lem:deg neq}
    For any $R>0$, there exists $\lambda_{a}\ge 0$, and $\lambda_{b}>\lambda_{1}$ such that 
    \begin{center}
    	$d(id-\lambda_{a}B((\cdot ,\cdot)), B_{R}(0,0),0)\ne d(id-\lambda_{b}B(\cdot), B_{R}(0,0),0)$.
    \end{center} 
\end{lemma}
\begin{proof}
Fix $R>0$. One may simply choose $\lambda_{a}=0$ and conclude
    \begin{equation}
    	d(id-\lambda_{a}B(\cdot), B_{R}(0,0),0)=d(id, B_{R}(0,0),0)=1.
    \end{equation}
In fact, the above equality holds for all $\lambda_{a}>0$ sufficiently small. To see this, if not, then there exist sequences $\left\{\lambda_{m}\right\}$ and $\left\{(u_{m},v_{m})\right\}$ satisfying $\lambda_{m}\rightarrow 0$, $\Vert u_{m}\Vert +\Vert v_{m}\Vert=R$, and
    \begin{equation}
    	(u_{m},v_{m})=\lambda_{m} B(u_{m},v_{m}).
    \end{equation}
The complete continuity of $B$ applied to the above equations implies there exists a solution $(u,v)$ with $\Vert u\Vert+\Vert v\Vert =R$, such that $u_{m}\rightarrow u,~v_{m}\rightarrow v$. However, combined with the fact that $\lambda_{m}\rightarrow 0$ implies $\Vert u\Vert+\Vert v\Vert =0$, it leads to a contradiction.

Next, we show that there exists a constant $\lambda_{b}>\lambda_{1}$ such that
    \begin{equation}\label{eq:11111}
    	d(id-\lambda_{b}B((\cdot ,\cdot)), B_{R}(0,0),0)=0.
    \end{equation}
We argue by contradiction. Let $(u_{0},v_{0})$ be a nonzero solution solution pair of \cref{eq:SKSE} corresponding to $\lambda_{b}$, then we have
    \begin{equation}\label{eq:22222}
    	(u_{0},v_{0})=\lambda_{b}B((u_{0},v_{0})).
    \end{equation}
Fix $R>0$. By the continuity of Leray-Schauder degree, we can choose $\varepsilon>0$ small such that
    \begin{center}
    	$d(id-\lambda_{b}B((\cdot ,\cdot)), B_{R}(0,0),0)=d(id-\lambda_{b}B((\cdot ,\cdot)+\varepsilon(u_{0},v_{0})), B_{R}(0,0),0)$.
    \end{center}
When \eqref{eq:11111} is not true, we obtain
    \begin{center}
    	$d(id-\lambda_{b}B((\cdot ,\cdot)+\varepsilon(u_{0},v_{0})), B_{R}(0,0),0)\ne 0$,
    \end{center} 
which implies the existence of $(\overline{u},\overline{v})$ such that 
    \begin{equation}\label{eq:33333}
        (\overline{u},\overline{v})=\lambda_{b}B((\overline{u},\overline{v})+\varepsilon(u_{0},v_{0})).
    \end{equation}
Recall the partial order induced by $M$ in $E$, we have $(\overline{u},\overline{v})\preccurlyeq (\overline{u},\overline{v})+\varepsilon(u_{0},v_{0})$. Since $B$ is monotone, we obtain 
    \begin{equation}
    	B((\overline{u},\overline{v}))\preccurlyeq B((\overline{u},\overline{v})+\varepsilon(u_{0},v_{0})).
    \end{equation}
then combining this with \eqref{eq:33333} gives us 
    \begin{equation}\label{eq:44444}
    	B((\overline{u},\overline{v}))\preccurlyeq \frac{(\overline{u},\overline{v})}{\lambda_{b}}.
    \end{equation} 
On the other hand, from $\varepsilon(u_{0},v_{0})\preccurlyeq (\overline{u},\overline{v})+\varepsilon(u_{0},v_{0})$, we have
    \begin{center}
    	$B(\varepsilon(u_{0},v_{0}))\preccurlyeq B((\overline{u},\overline{v})+\varepsilon(u_{0},v_{0}))$;
    \end{center}
using \eqref{eq:33333} again, we have
    \begin{equation}
    	\lambda_{b} B(\varepsilon(u_{0},v_{0}))\preccurlyeq (\overline{u},\overline{v}).
    \end{equation}
By the above partial order relation, \eqref{eq:22222} and the homogeneity of $B$, we have
    \begin{equation}\label{eq:55555}
    	\frac{\lambda_{b}\varepsilon(u_{0},v_{0})}{\lambda_{1}}\preccurlyeq (\overline{u},\overline{v}).
    \end{equation}
Now operate $B$ on both sides of \eqref{eq:55555}, we have 
    \begin{equation}
    	\frac{\lambda_{b}\varepsilon B((u_{0},v_{0}))}{\lambda_{1}}\preccurlyeq B((\overline{u},\overline{v})).
    \end{equation}
Combining this with \eqref{eq:22222} and \eqref{eq:44444}, we deduce
    \begin{equation}
    	\frac{\lambda_{b}^{2}\varepsilon(u_{0},v_{0})}{\lambda_{1}^{2}}\preccurlyeq (\overline{u},\overline{v}).
    \end{equation}
Noticing this together with \eqref{eq:55555}, one can prove by induction that
    \begin{center}
    	$\frac{\lambda_{b}^{n}\varepsilon(u_{0},v_{0})}{\lambda_{1}^{n}}\preccurlyeq (\overline{u},\overline{v}),\quad \forall n\in \mathbb{N}$.
    \end{center}
So,
    \begin{center}
    	$(u_{0},v_{0})\preccurlyeq(\frac{\lambda_{1}}{\lambda_{b}})^{n}\cdot\frac{(\overline{u},\overline{v})}{\varepsilon},\quad \forall n\in \mathbb{N}$.
    \end{center}
Letting $n\rightarrow +\infty$, from $\lambda_{b}>\lambda_{1}>0$ we obtain $(u_{0},v_{0})\preccurlyeq (0,0)$. Thus $-(u_{0},v_{0})\in M$, giving $(u_{0},v_{0})\in M\cap (-M)=(0,0)$, contradiction with $(u_{0},v_{0})\ne (0,0)$.
\end{proof} 

\begin{prop}\label{prop:SKEBifur1} Assume $(A_{1})$ and $(A_{2})$ hold, then $(\mu ,(0,0))$ is a bifurcation point for \eqref{eq:h} if and only if $\mu =\lambda_{1}$, furthermore the continuum $\mathcal{C}$ emanating from $(\lambda_{1},(0,0))$ is unbounded. 
\end{prop}
\begin{proof}

Necessary: Suppose $(\mu ,(0,0))$ is a bifurcation point for \eqref{eq:h}. Then there exists a sequence $(\lambda_{n},(u_{n},v_{n}))\rightarrow(\mu ,(0,0))$, such that $\Vert u_{n}\Vert +\Vert v_{n}\Vert \ne 0$ for all $n$ and $(\lambda_{n},(u_{n},v_{n}))$ satisfy
    \begin{equation}
    	\begin{cases}
    		S_{k}(D^{2}u_{n})=\vert \lambda_{n} v_{n}\vert ^{k} +\vert \lambda_{n}\vert ^{k}f(u_{n},v_{n})\qquad& in\quad\Omega \\
    		S_{k}(D^{2}v_{n})=\vert \lambda_{n} u_{n}\vert ^{k}+\vert \lambda_{n}\vert ^{k}g(u_{n},v_{n})\qquad& in\quad\Omega \\
    		u_{n}=v_{n}=0                 \qquad& on\quad \partial\Omega
    	\end{cases}.
    \end{equation}
Divide each equations above by $(\Vert u_{n}\Vert +\Vert v_{n}\Vert )^{k}$ and denote
    \begin{equation}
    	\begin{aligned}
    		&\widetilde{u}_{n}=\frac{u_{n}}{\Vert u_{n}\Vert +\Vert v_{n}\Vert },\quad\widetilde{v}_{n}=\frac{v_{n}}{\Vert u_{n}\Vert +\Vert v_{n}\Vert },\\
    		&\widetilde{f}_{n}=\frac{f(u_{n},v_{n})}{(\Vert u_{n}\Vert +\Vert v_{n}\Vert )^{k}},\quad\widetilde{g}_{n}=\frac{g(u_{n},v_{n})}{(\Vert u_{n}\Vert +\Vert v_{n}\Vert )^{k}}.\nonumber
    	\end{aligned}
    \end{equation}
Then we obtain
    \begin{equation}\label{eq:SKBBB}
    	\begin{cases}
    		S_{k}(D^{2}\widetilde{u}_{n})=\vert \lambda_{n} \widetilde{v}_{n}\vert ^{k} +\vert \lambda_{n}\vert ^{k}\widetilde{f}_{n}(u_{n},v_{n})\qquad& in\quad\Omega \\
    		S_{k}(D^{2}\widetilde{v}_{n})=\vert \lambda_{n} \widetilde{u}_{n}\vert ^{k} +\vert \lambda_{n}\vert ^{k}\widetilde{g}_{n}(u_{n},v_{n})\qquad& in\quad\Omega \\
    		\widetilde{u}_{n}=\widetilde{v}_{n}=0                 \qquad& on\quad \partial\Omega
    	\end{cases}.
    \end{equation}
Note $\Vert u_{n}\Vert +\Vert v_{n}\Vert\ne 0 $ and  $u_{n},~v_{n}$ are $k$-admissible functions with zero boundary date, we have $\vert u_{n}\vert +\vert v_{n}\vert\ne 0$ for any $x\in \Omega$. Thus, for $x\in \Omega$, we have 
    \begin{center}
    	$0\le \widetilde{f}_{n}=\frac{f(u_{n},v_{n})}{(\vert u_{n}\vert +\vert v_{n}\vert)^{k}}.(\frac{\vert u_{n}\vert +\vert v_{n}\vert }{\Vert u_{n}\Vert +\Vert v_{n}\Vert})^{k}\le \frac{f(u_{n},v_{n})}{(\vert u_{n}\vert +\vert v_{n}\vert)^{k}}$.
    \end{center}
Notice that $(u_{n},v_{n})\rightarrow(0,0)$ in $C(\overline{\Omega})\times C(\overline{\Omega})$, we deduce from the above inequalities and $(A_{2})$ that $\widetilde{f}_{n}\rightarrow0$ uniformly for $x\in \Omega$ as $n\rightarrow+\infty$. Combining this with the facts $\Vert \widetilde{v}_{n}\Vert\le 1$ and $\lambda_{n}\rightarrow\mu $, we see that $\left\{\vert \lambda_{n}\widetilde{v}_{n}\vert ^{k}+\vert \lambda_{n}\vert ^{k}\widetilde{f}_{n}\right\}$ is bounded in $C(\overline{\Omega})$. Hence by the compactness of $T_{k}$, we obtain a sub-sequence of $\left\{\widetilde{u}_{n}\right\}$, still denoted $\left\{\widetilde{u}_{n}\right\}$ such that $\widetilde{u}_{n}\rightarrow u_{*}$ for some $u_{*}\in C(\overline{\Omega})$. Similarly, $\widetilde{v}_{n}\rightarrow v_{*}$. Then, by the continuity of $T_{k}$, we infer from \cref{eq:SKBBB} that
    \begin{center}
    	$\begin{cases}
    		u_{*}=T_{k}(\vert \mu v_{*}\vert ^{k})\\
    		v_{*}=T_{k}(\vert \mu u_{*}\vert ^{k})
    	\end{cases}.$
    \end{center}
    
We claim $(u_{*},v_{*})\ne (0,0)$, indeed
    \begin{equation*}
    	\begin{aligned}
    		&\Vert u_{*}\Vert =\underset{n\rightarrow+\infty}{lim}\frac{\Vert u_{n}\Vert}{\Vert u_{n}\Vert +\Vert v_{n}\Vert }, \\
    		&\Vert v_{*}\Vert =\underset{n\rightarrow+\infty}{lim}\frac{\Vert v_{n}\Vert }{\Vert u_{n}\Vert +\Vert v_{n}\Vert},
    	\end{aligned}
    \end{equation*}
which yield
    \begin{center}
    	$\Vert u_{*}\Vert +\Vert v_{*}\Vert =\underset{n\rightarrow+\infty}{\lim }\frac{\Vert u_{n}\Vert +\Vert v_{n}\Vert}{\Vert u_{n}\Vert +\Vert v_{n}\Vert }=1$.
    \end{center}
By \cref{coro:SKSE} we reach that $\mu =\lambda_{1}$.

Sufficiency: We need to find $\lambda_{a}$ and $\lambda_{b}$ which $\lambda_{a}<\lambda_{1}<\lambda_{b}$ such that $(u,v)=(0,0)$ is a isolated solution for \eqref{eq:h}. In fact, $(0,0)$ is isolated of \eqref{eq:h} for $\mu \ne \lambda_{1}$ by above. Let
    \begin{center}
    	$\begin{cases}
    		\overline{H}_{1}(t,(u,v))=T_{k}( \vert \mu v\vert ^{k}+t\mu^{k} f(u,v)) \\
    		\overline{H}_{2}(t,(u,v))=T_{k}( \vert\mu u\vert ^{k}+t\mu^{k} g(u,v))	
    	\end{cases}$,
    \end{center}
then $\overline{H}_{t}=(\overline{H}_{1},\overline{H}_{2})$ is a degree preserving homotopy between the operator 
    \begin{center}
    	$\begin{cases}
    		T_{k}(\vert \mu v\vert ^{k}+t\mu^{k} f(u,v))\\
    		T_{k}(\vert \mu u\vert ^{k}+t\mu^{k} g(u,v))
    	\end{cases}$
    	and\quad
    	$\begin{cases}
    		T_{k}(\vert \mu v\vert ^{k})\\
    		T_{k}(\vert \mu u\vert ^{k})
    	\end{cases}.$
    \end{center}
Otherwise, there exist $\left\{t_{m}\right\}$ and $\left\{(u_{m},v_{m})\right\}\rightarrow(0,0),~\left\{(u_{m},v_{m})\right\}\ne (0,0)$ satisfy 
    \begin{center}
    	$\begin{cases}
    		S_{k}(D^{2}u_{m})=\vert \mu v_{m}\vert ^{k}+t_{m}\mu^{k}f(u_{m},v_{m})\\ 
    		S_{k}(D^{2}v_{m})=\vert \mu u_{m}\vert ^{k}+t_{m}\mu^{k}g(u_{m},v_{m})
    	\end{cases}.$
    \end{center}
Similarly to above arguments, we will get contradiction. So in order to find $\lambda_{a}<\lambda_{1}<\lambda_{b}$ such that
    \begin{center}
    	$d(id-H(\lambda_{a}),B_{R}(0,0),0)\ne d(id-H(\lambda_{b}),B_{R}(0,0),0) $,
    \end{center}
where $B_{R}(0,0)$ is a sufficiently small isolating neighborhood of the trivial solution. And it is sufficient to find $\lambda_{a}<\lambda_{1}<\lambda_{b}$ such that
    \begin{center}
    	$d(id-\lambda_{a}B(\cdot),B_{R}(0,0),0)\ne d(id-\lambda_{b}B(\cdot),B_{R}(0,0),0)$.
    \end{center}
However, this is precisely the conclusion of \cref{lem:deg neq}, which is valid for any $R>0$. 

Therefore, if we let $\mathcal{S}$ be defined by
    \begin{center}
    	$\mathcal{S}=\overline{\left\{(\lambda,(u,v)):(\lambda,(u,v)) ~is ~a~ solution ~of ~\cref{eq:SKaux} ~with~(u,v)\ne (0,0)\right\}}\cup (\left[\lambda_{a},\lambda_{b}\right]\times \left\{0\right\})$,
    \end{center}
then, by \cref{lem:GB}, there exists a connected component $\mathcal{C}$ of $\mathcal{S}$ containing $\left[\lambda_{a},\lambda_{b}\right]\times \left\{0\right\}$ which means either
    \par 
    (1) $\mathcal{C}$ is unbounded in $\mathbb{R}\times E$, or 
    \par 
    (2) $\mathcal{C}\cap  [(\mathbb{R}\backslash[a,b])\times \left\{0\right\}]\ne \emptyset$.\\
Since $(u,v)=(0,0)$ is the unique solution corresponding to $\lambda=0$, the continuum $\mathcal{C}$ cannot cross the $\lambda$=0 axis, other than at the trivial crossing, when $(u,v)=(0,0)$. Furthermore, since $(u,v)=(0,0)$ is an isolated solution for \eqref{eq:h} for all $\lambda\ge 0$ with $\lambda\ne \lambda_{1}$, the alternative (2) above may not hold. Therefore we may conclude $\mathcal{C}$ is unbounded in $\mathbb{R}\times E$.
\end{proof}

In order to study the asymptotic bifurcation phenomena, we need another assumption.

${\bf (A_{3})}$ $\frac{f(s,t)}{(\vert s\vert +\vert t\vert )^{k}}\rightarrow 0$ and $\frac{g(s,t)}{(\vert s\vert +\vert t\vert )^{k}}\rightarrow 0$ as $\vert s\vert +\vert t\vert \rightarrow \infty$.
    
\begin{prop}\label{prop:SKEBifur2} Assume $(A_{1})$ and $(A_{3})$ hold, then $\mu $ is an asymptotic bifurcation value for \eqref{eq:h} if and only if $\mu =\lambda_{1}$.
\end{prop}

\begin{proof}

Necessity. Let $\mu$ be a bifurcation value for \eqref{eq:h}. Then there exists a sequence $(\lambda_{m},(u_{m},v_{m}))$, such that $\lambda_{m}\rightarrow \mu $, $\Vert u_{m}\Vert +\Vert v_{m}\Vert \rightarrow \infty$ for $m\rightarrow+\infty$ and $(\lambda_{m},(u_{m},v_{m}))$ satisfy
    \begin{equation*}
    	\begin{cases}
    		S_{k}(D^{2}u_{m})=\vert \lambda_{m} v_{m}\vert ^{k} +\vert \lambda_{m}\vert ^{k}f(u_{m},v_{m})\qquad& in\quad\Omega \\
    		S_{k}(D^{2}v_{m})=\vert \lambda_{m} u_{m}\vert ^{k}+\vert \lambda_{m}\vert ^{k}g(u_{m},v_{m})\qquad& in\quad\Omega \\  
    		u_{m}=v_{m}=0                 \qquad& on\quad \partial\Omega
    	\end{cases}.
    \end{equation*}
Divide each equation in above by $(\Vert u_{m}\Vert +\Vert v_{m}\Vert )^{k}$ and denote
    \begin{equation}
    	\begin{aligned}
    		&\widetilde{u}_{m}=\frac{u_{m}}{\Vert u_{m}\Vert +\Vert v_{m}\Vert },\quad\widetilde{v}_{m}=\frac{v_{m}}{\Vert u_{m}\Vert +\Vert v_{m}\Vert },\\
    		&\widetilde{f}_{m}=\frac{f(u_{m},v_{m})}{(\Vert u_{m}\Vert +\Vert v_{m}\Vert )^{k}},\quad\widetilde{g}_{m}=\frac{g(u_{m},v_{m})}{(\Vert u_{m}\Vert +\Vert v_{m}\Vert )^{k}}.\nonumber
    	\end{aligned}
    \end{equation}
Then we obtain
    \begin{equation*}
    	\begin{cases}
    		S_{k}(D^{2}\widetilde{u}_{m})=\vert \lambda_{m} \widetilde{v}_{m}\vert ^{k} +\vert \lambda_{m}\vert ^{k}\widetilde{f}_{m}(u_{m},v_{m})\qquad& in\quad\Omega \\
    		S_{k}(D^{2}\widetilde{v}_{m})=\vert \lambda_{m} \widetilde{u}_{m}\vert ^{k} +\vert \lambda_{m}\vert ^{k}\widetilde{g}_{m}(u_{m},v_{m})\qquad& in\quad\Omega \\
    		\widetilde{u}_{m}=\widetilde{v}_{m}=0                 \qquad& on\quad \partial\Omega
    	\end{cases}.
    \end{equation*}
By condition $(A_{3})$, for each $\varepsilon>0$, there exists $M_{0}>0$,such that if $\vert (s,t)\vert :=\vert s\vert +\vert t\vert >M_{0}$, then
    \begin{equation*}
    	\frac{f(s,t)}{(\vert s\vert +\vert t\vert  )^{k}}<\varepsilon.
    \end{equation*}
For this $M_{0}>0$, denote $f^{*}:=max_{\vert (s,t)\vert\le M_{0}}f(s,t)$, then for large $m$,
    \begin{equation*}
    	\frac{f^{*}}{(\Vert u_{m}\Vert +\Vert v_{m}\Vert )^{k}}<\varepsilon.
    \end{equation*}
By the above two inequalities with $(A_{1})$, we deduce that for $m$ sufficiently large,
    \begin{center}
    	$0\le \widetilde{f}_{m}<\varepsilon,\quad\forall x\in \Omega$.
    \end{center}
So we have $\widetilde{f}_{m}\rightarrow 0$ as $m\rightarrow+\infty$. Similarly, $\widetilde{g}_{m}\rightarrow 0$ uniformly for $x\in \Omega$, as $m\rightarrow+\infty$.
By mimicking the counterpart in the proof of \cref{prop:SKEBifur1}, one is ready to reach $\mu =\lambda_{1}$. 

Sufficiency: We need to find $\lambda_{a}$ and $\lambda_{b}$ which $\lambda_{a}<\lambda_{1}<\lambda_{b}$ such that solutions to \eqref{eq:h} satisfy
    \begin{equation}\label{eq:66666}
    	(T_{k}( \vert\lambda_{a} v \vert ^{k}+\lambda_{a}^{k} f(u,v)),T_{k}( \vert\lambda_{a} u \vert ^{k}+\lambda_{a}^{k} g(u,v)))\ne (u,v)\ne (T_{k}( \vert\lambda_{b} v \vert ^{k}+\lambda_{b}^{k} f(u,v)),T_{k}( \vert\lambda_{b} u \vert ^{k}+\lambda_{b}^{k} g(u,v)))
    \end{equation}
for all $u,v\in E$ with $\Vert u\Vert +\Vert v\Vert \ge M$, and
    \begin{equation}\label{eq:77777}
    	d(id-H(\lambda_{a}),B_{R}(0,0),0)\ne d(id-H(\lambda_{b}),B_{R}(0,0),0),
    \end{equation}
for $R>M$. Condition \eqref{eq:66666} is readily satisfied by choosing any constants $\lambda_{a}$ and $\lambda_{b}$ such that $\lambda_{a}\ne \lambda_{1}\ne \lambda_{b}$, as we have already shown that $\lambda_{1}$ is the only positive asymptotic bifurcation value for \eqref{eq:h}. Furthermore, by using a homotopy argument, dual to the case considered in \cref{lem:deg neq}, one sees that \eqref{eq:77777} is equivalent to
    \begin{center}
    	$d(id-\lambda_{a}B(\cdot),B_{R}(0,0),0)\ne d(id-\lambda_{b}B(\cdot),B_{R}(0,0),0)$,
    \end{center}
for all $R>0$, sufficiently large. However, this is precisely the conclusion of \cref{prop:SKEBifur1}, which is valid for any $R>0$. Therefore, by \cref{lem:GAB}, there exists a continuum $\mathcal{C}^{\infty}$ to \cref{eq:SKaux} that is unbounded in $[\lambda_{a},\lambda_{b}]\times E$.
\end{proof}

\subsection{A-priori Estimate}

Firstly, we need some result of a-priori estimate which inspired by \cite{chou_variational_2001} and \cite{wang_existence_1992}.
    \begin{equation}\label{eq:CSK3}
    	\begin{cases}
    		S_{k}(D^{2}u)=f(u,v)\qquad& in\quad\Omega \\  
    		S_{k}(D^{2}v)=g(u,v)\qquad& in\quad\Omega \\  
    		u=v=0               \qquad& on\quad \partial\Omega
    	\end{cases}
    \end{equation}
where $\Omega$ is a $(k$-$1)$-convex domain. 

\begin{prop} \label{prop:priori}
Consider \cref{eq:CSK3} with the following hold.
    \begin{center}
    	$\underset{\vert s\vert +\vert t \vert \rightarrow \infty}{lim}\frac{f(s,t)}{\vert t\vert ^{k}}<\lambda_{1}^{k},\quad\underset{\vert u\vert +\vert v \vert \rightarrow \infty}{lim}\frac{g(s,t)}{\vert s\vert ^{k}}<\lambda_{1}^{k}$.
    \end{center} 
Then there exists $C>0$ depending on $n,~k$ and $\Omega$ such that for any $k$-admissible solution of \cref{eq:CSK3},
    \begin{center}
    	$\Vert u\Vert +\Vert v\Vert \le C$.
    \end{center}
\end{prop}   

\begin{proof}
We define that
    \begin{center}
    	$F(u,v)=\int_{0}^{u}f(s,v)~ds,~G(u,v)=\int_{0}^{v}g(u,t)~dt$
    \end{center}
Then, by the condition, we have
    \begin{equation*}
    	\begin{aligned}
    		&F(u,v)\le C_{1}+(1-\theta_{1})\lambda_{1}^{k}\vert v\vert ^{k}\vert u\vert  \\
    		&G(u,v)\le C_{2}+(1-\theta_{2}\lambda_{1}^{k}\vert u\vert ^{k}\vert v\vert .
    	\end{aligned}
    \end{equation*}
Furthermore, we can get that
    \begin{center}
    	$f(u,v)\le (K_{1}-(\lambda_{1}-\theta_{1})v)^{k},\quad g(u,v)\le (K_{2}-(\lambda_{1}-\theta_{2})u)^{k}$
    \end{center}
for some constants $K_{1}=K_{1}(\theta_{1},C_{1})>0$, $K_{2}=K_{2}(\theta_{2},C_{2})>0$. If the conclusion of the proposition is not true, there are sequence $\left\{f_{m}\right\}$ and $\left\{g_{m}\right\}$ satisfying the  inequalities above such that the \cref{eq:CSK3} with $f=f_{m}$, $g=g_{m}$ has a solution $(u_{m},v_{m})$ with
    \begin{center}
    	$M_{m}=\underset{\Omega}{sup}~(\vert u_{m}\vert +\vert v_{m}\vert )\rightarrow\infty$ as $m\rightarrow\infty$.
    \end{center}
Let $w_{m}=u_{m}/M_{m},~z_{m}=v_{m}/M_{m}$, then we have
    \begin{equation}\label{eq:CSK4}
    	\begin{cases}
    		S_{k}(D^{2}w_{m})=\frac{f_{m}(u_{m},v_{m})}{M_{m}^{k}}\le \frac{(K_{1}-(\lambda_{1}-\theta_{1})v_{m})^{k}}{M_{m}^{k} }\qquad&in\quad\Omega\\ 
    		S_{k}(D^{2}z_{m})=\frac{g_{m}(u_{m},v_{m})}{M_{m}^{k}}\le \frac{(K_{2}-(\lambda_{1}-\theta_{2})u_{m})^{k}}{M_{m}^{k} }\qquad&in\quad\Omega\\ 
    		w_{m}=z_{m}=0 \qquad& on \quad\partial\Omega
    	\end{cases},
    \end{equation}
and $w_{m}\rightarrow w$, $z_{m}\rightarrow z$. By the weak convergence of the Hessian measure, $(w,z)$ is a super-solution of 
    \begin{equation}\label{eq:CSK5}
    	\begin{cases}
    		S_{k}(D^{2}u)= \vert C(\lambda_{1}-\theta_{1})v\vert^{k}\qquad&in\quad\Omega\\ 
    		S_{k}(D^{2}v)= \vert C'(\lambda_{1}-\theta_{2})u\vert^{k}\qquad&in\quad\Omega\\ 
    		u=v=0 \qquad& on \quad\partial\Omega
    	\end{cases}
    \end{equation}
Furthermore, we have $C,~C'\le 1$ by the definition of $M_{m}$. Let $a,~b>1$ be sufficiently large such that $u=a\phi <w$, $v=b\psi <z$, where $(\phi,\psi)$ is the solution of \cref{eq:SKSE}. Then $(u,v)$ and $(w,z)$ are, respectively, a sub-solution and a super-solution of \cref{eq:CSK5}. It follows that there is a solution $(\phi _{1},\psi_{1})$ of \cref{eq:CSK5} satisfying $u\le \phi_{1}\le w$, $v\le \psi_{1}\le z$. However, this is in conflict with the uniqueness of the first eigenvalue.
\end{proof} 

\begin{prop}  \label{pro:priori2}
Consider 
    \begin{equation}\label{eq:CSK6}
    	\begin{cases}
    		S_{k}(D^{2}u)=f(x,u,v)\qquad&in \quad\Omega\\ 
    		S_{k}(D^{2}v)=g(x,u,v)\qquad&in \quad\Omega\\ 
    		u=v=0\qquad& on\quad\partial\Omega
    	\end{cases}
    \end{equation}
if there exist non-decreasing functions $\Phi(s,t)\ge 0$, $\Psi(s,t)\ge 0$ such that
    \begin{center}
    	$f(x,s,t)\ge \Phi(s,t)\quad \forall x\in \Omega,~s,~t\le 0$,\quad$g(x,s,t)\ge \Psi(s,t)\quad \forall x\in \Omega,~s,~t\le 0$,
    \end{center}
and
    \begin{equation}\label{eq:CSK7}
    	\underset{\vert s\vert +\vert t\vert \rightarrow\infty}{lim}\Phi (s,t)/\vert t\vert ^{k}=+\infty,\quad\underset{\vert s\vert +\vert t\vert \rightarrow\infty}{lim}\Psi (s,t)/\vert s\vert ^{k}=+\infty.
    \end{equation}
Then let $(u,v)\in C^{0,1}(\overline{\Omega})\times C^{0,1}(\overline{\Omega})$ be an admissible solution of \cref{eq:CSK6}, we have $\Vert u\Vert_{0}+\Vert v\Vert_{0}\le M$, where $M$ depends only on $\Phi,~\Psi,~k,~\Omega$.
\end{prop}
    
\begin{proof} For any fixed $\tau>0$, denote $\Omega^{\tau}=\left\{\tau x:x\in \Omega\right\}$. By the uniqueness of eigenvalues we have $\lambda_{1}(\Omega^{\tau})=\tau^{-2}\lambda_{1}(\Omega)$. Without loss of generality we may suppose the origin $O\in \Omega$, which implies $\Omega^{t}\subset \Omega$ for any $\tau<1$.

If the conclusion of the proposition is not true, then there exists a sequence of functions $f_{j}(x,u,v),~g_{j}(x,u,v)$ with 
    \begin{center}
    	$f_{j}(x,s,t)\ge \Phi(s,t)\quad \forall x\in \Omega,~s,~t\le 0$,\quad$g_{j}(x,s,t)\ge \Psi(s,t)\quad \forall x\in \Omega,~s,~t\le 0$,
    \end{center}
so that the sequence of solutions $(u_{j},v_{j})$ of 
    \begin{center}
    	$\begin{cases}
    		S_{k}(D^{2}u_{j})=f_{j}(x,u_{j},v_{j})\qquad&in \quad\Omega\\ 
    		S_{k}(D^{2}v_{j})=g_{j}(x,u_{j},v_{j})\qquad&in \quad\Omega\\ 
    		u_{j}=v_{j}=0\qquad& on\quad\partial\Omega
    	\end{cases}$
    \end{center}
satisfies $M_{j}=\Vert u_{j}\Vert+\Vert v_{j}\Vert=-(\underset{{x\in \Omega}}{inf}u_{j}+\underset{x\in \Omega}{inf}v_{j})\rightarrow+\infty$. From \cref{eq:CSK7} there exist $\overline{\lambda}>>\lambda_{1}$ and $\tau \in (0,1)$ such that $\underset{\vert s\vert +\vert t\vert \rightarrow\infty}{lim}\Phi (s,t)/(\vert s\vert +\vert t\vert) ^{k}\ge \overline{\lambda}^{k}$, $\underset{\vert s\vert +\vert t\vert \rightarrow\infty}{lim}\Psi (s,t)/(\vert s\vert +\vert t\vert) ^{k}\ge \overline{\lambda}^{k}$ and $\lambda_{1}(\Omega^{\tau})=\tau^{-2}\lambda_{1}(\Omega)<\overline{\lambda}$. Since $u_{j}(x)+v_{j}(x)\rightarrow-\infty$ uniformly in $\Omega^{\tau}$, we have 
    \begin{center}
    	$\begin{cases}
    		S_{k}(D^{2}u_{j})\ge \vert \overline{\lambda} v_{j}\vert^{k}\qquad&in \quad\Omega^{\tau}\\ 
    		S_{k}(D^{2}v_{j})\ge \vert \overline{\lambda} u_{j}\vert^{K}\qquad&in \quad\Omega^{\tau}
    	\end{cases}$
    \end{center}
provided $j$ is sufficient large. Fix such a $j$. Let $(\phi_{\tau},\psi_{\tau})$ be the eigen-function of the $k$-Hessian operator on $\Omega^{\tau}$. Replacing $(\phi_{\tau},\psi_{\tau})$ by $\ell(\phi_{\tau},\psi_{\tau})$ for some small $\ell$ we may suppose $\phi_{\tau}(x)>u_{j}(x)$, $\psi_{\tau}(x)>v_{j}(x)$ in $\Omega^{t}$. Thus $(\phi_{\tau},\psi_{\tau})$ and $(u_{j}(x),v_{j}(x))$ are the super-solution and sub-solution of the Dirichlet problem
    \begin{equation}\label{eq:CSK8}
    	\begin{cases}
    		S_{k}(D^{2}u_{j})=\vert \overline{\lambda} v_{j}\vert^{k}\qquad&in \quad\Omega^{\tau}\\ 
    		S_{k}(D^{2}v_{j})=\vert \overline{\lambda} u_{j}\vert^{k}\qquad&in \quad\Omega^{\tau}\\ 
    		u_{j}=v_{j}=0\qquad& on\quad\partial\Omega^{\tau}
    	\end{cases}
    \end{equation}
respectively. Therefore this is a solution $(u,v)$ of \cref{eq:CSK8} which satisfies $\phi_{\tau}(x)\ge u(x)\ge u_{j}(x)$ and $\psi_{\tau}(x)\ge v(x)\ge v_{j}(x)$. This means both $\overline{\lambda}$ and $\tau^{-2}\lambda_{1}$ are the eigenvalues of the $k$-Hessian operator on $\Omega^{\tau}$, which is contradictory to the uniqueness of eigenvalues.
\end{proof}

\subsection{Proof of \cref{thm:SKS1}}
In the following, we prove \cref{thm:SKS1}.

\begin{theorem}[\cref{thm:SKS1}]
    Let $g_{0}=\underset{\vert s+t\vert\rightarrow0}{lim}\frac{g(s,t)}{\vert t\vert}$, $h_{0}=\underset{\vert s+t\vert \rightarrow0}{lim}\frac{h(s,t)}{\vert s\vert }$, and $g_{\infty}=\underset{\vert s+t\vert\rightarrow+\infty}{lim}\frac{g(s,t)}{\vert t\vert}$,
    $h_{\infty}=\underset{\vert s+t\vert \rightarrow+\infty}{lim}\frac{h(s,t)}{\vert s\vert}$, then\\
    (1) if $f_{0}:=g_{0}=h_{0}\in (0,+\infty)$, $f_{\infty}:=g_{\infty}=h_{\infty}\in (0,\infty)$, and $g_{0}\ne g_{\infty}$ then \cref{eq:SKS} has at least one $k$-admissible solution for $\lambda\in (min\left\{\frac{\lambda_{0}}{f_{0}},\frac{\lambda_{0}}{f_{\infty}}\right\},max\left\{\frac{\lambda_{0}}{f_{0}},\frac{\lambda_{0}}{f_{\infty}}\right\})$.\\
    (2) if $f_{0}:=g_{0}=h_{0}\in (0,+\infty)$, $f_{\infty}:=g_{\infty}=h_{\infty}=0$, then  \cref{eq:SKS} has at least one nontrivial solution for every $\lambda\in (\lambda_{0}/f_{0},+\infty)$.\\
    (3) if $f_{0}:=g_{0}=h_{0}\in (0,+\infty)$, $f_{\infty}:=g_{\infty}=h_{\infty}=+\infty$, then  \cref{eq:SKS} has at least one nontrivial solution for every $\lambda\in (0,\lambda_{0}/f_{0})$.\\
    (4) if $f_{0}:=g_{0}=h_{0}=0$, $f_{\infty}:=g_{\infty}=h_{\infty}\in (0,+\infty)$, then  \cref{eq:SKS} has at least one nontrivial solution for every $\lambda\in (\lambda_{0}/f_{\infty},+\infty)$.\\
    (5) if $f_{0}:=g_{0}=h_{0}=0$, $f_{\infty}:=g_{\infty}=h_{\infty}=\infty$, then  \cref{eq:SKS} has at least one nontrivial solution for every $\lambda\in (0,+\infty)$.\\
    (6) if $f_{0}:=g_{0}=h_{0}=+\infty$, $f_{\infty}:=g_{\infty}=h_{\infty}\in (0,+\infty)$, then  \cref{eq:SKS} has at least one nontrivial solution for every $\lambda\in (0,\lambda_{0}/f_{\infty})$.\\
    (7) if $f_{0}:=g_{0}=h_{0}=+\infty$, $f_{\infty}:=g_{\infty}=h_{\infty}=0$, then  \cref{eq:SKS} has at least one nontrivial solution for every $\lambda\in (0,+\infty)$.  
\end{theorem} 
  
\begin{proof}
(1) Let $\overline{g}(s,t)$,~$\overline{h}(s,t)$ satisfy
    $\underset{\vert s\vert +\vert t\vert \rightarrow0}{lim}\frac{\overline{g}(s,t)}{(\vert s\vert +\vert t\vert )^{k}}=0$ and $\underset{\vert s\vert +\vert t\vert \rightarrow0}{lim}\frac{\overline{h}(s,t)}{(\vert s\vert +\vert t\vert )^{k}}=0$ such that
    \begin{equation*}
    	\begin{aligned}
    		&g^{k}(s,t)=g_{0}^{k} \vert t\vert ^{k}+\overline{g}(s,t),\\
    		&h^{k}(s,t)=h_{0}^{k} \vert t\vert ^{k}+\overline{h}(s,t).
    	\end{aligned}
    \end{equation*}
Then \cref{eq:CSK1} is equivalent to
    \begin{equation}\label{eq:CSK9}
    	\begin{cases}
    		S_{k}(D^{2}u)=\lambda^{k}(f_{0}^{k} \vert v\vert ^{k}+\overline{g}(-u,-v))\qquad&in \quad\Omega \\ 
    		S_{k}(D^{2}v)=\lambda^{k}(g_{0}^{k} \vert u\vert ^{k}+\overline{h}(-u,-v))\qquad&in \quad\Omega \\ 
    		u=v=0\qquad&on \quad\partial\Omega
    	\end{cases}
    \end{equation}
By \cref{prop:SKEBifur1}, there exists an unbounded continuum $\mathcal{C}$ solutions of  \cref{eq:CSK9} emanating from $(\lambda_{0}/ f_{0},(0,0))$ such that $\mathcal{C}\subset(\mathbb{R}\times E)\cup \left\{(\lambda_{0}/ f_{0},(0,0))\right\}$. Next, we show that $\mathcal{C}$ links $(\lambda_{0}/ f_{0},(0,0))$ to $(\lambda_{0}/ f_{\infty},\infty)$ in $\mathbb{R}\times E$. Let $(\lambda_{n},(u_{n},v_{n}))\in \mathcal{C}$ satisfy
    \begin{center}
    	$\lambda_{n}+\Vert u_{n}\Vert +\Vert v_{n}\Vert \rightarrow +\infty$
    \end{center}
as $n\rightarrow+\infty$. Since $(0,0)$ is the only solution of \eqref{eq:h} for $\lambda=0$, then we have $\mathcal{C}\cap (\left\{0\right\}\times E)=\emptyset$. So we have $\lambda_{n}>0$ for all $n\in \mathbb{N}$.

We claim that there exists a constant $M>0$ such that 
    \begin{center}
    	$\lambda_{n}\in \left(0,M\right]$
    \end{center}
for $n$ large enough.

In fact, if we stand on the contrary, supposing that $\underset{n\rightarrow+\infty}{lim}\lambda_{n}=+\infty$, since $f_{0},~f_{\infty}\in (0,+\infty)$, there exists a positive constant $\sigma$ such that
    \begin{center}
    	$g^{k}(-u_{n},-v_{n})/\vert v_{n}\vert^{k} \ge \sigma,\quad h^{k}(-u_{n},-v_{n})/\vert u_{n}\vert^{k} \ge \sigma~for ~any ~n\in \mathbb{N}$.
    \end{center}
So we have $\underset{n\rightarrow+\infty}{lim}\frac{\lambda_{n}^{k}g^{k}(-u_{n},-v_{n})}{\vert v_{n}\vert ^{k}}=\underset{n\rightarrow+\infty}{lim}\frac{\lambda_{n}^{k}h^{k}(-u_{n},-v_{n})}{\vert u_{n}\vert ^{k}}=+\infty$, by \cref{pro:priori2} if necessary, for $n$ large we have $\Vert u_n\Vert +\Vert v_{n}\Vert\le C$ where $C>0$ is a constant and $\frac{\lambda_{n}^{k}g^{k}(-u_{n},-v_{n})}{\vert v_{n}\vert ^{k}}>\overline{\lambda}^{k}>\lambda_{1}^{k}$ and $\frac{\lambda_{n}^{k}g^{k}(-u_{n},-v_{n})}{\vert u_{n}\vert ^{k}}>\overline{\lambda}^{k}>\lambda_{1}^{k}$ for some constant $\overline{\lambda}> \lambda_{1}$. Let $(w,z)$ satisfies 
    \begin{center}
    	$\begin{cases}
    		S_{k}(D^{2}w)=\lambda_{1}^{k}\vert z\vert ^{k}\qquad &in \quad\Omega\\ 
    		S_{k}(D^{2}z)=\lambda_{1}^{k}\vert w\vert ^{k}\qquad &in \quad\Omega\\ 
    		w=z=0\qquad&on\quad\partial\Omega
    	\end{cases}$
    \end{center}
By scaling if necessary, for $n$ large enough, we can assume $u_{n}(x)<w(x)$ and $v_{n}(x)<z(x)$ for all $x\in \Omega$. Let $\delta^{*},\delta^{**}>0$ be the maximal such that $u_{n}-\delta^{*}w\le 0$, $v_{n}-\delta^{**}z\le 0$ in $\Omega$. Let $\omega=\delta^{*}w$ and $\tau =\delta^{**}z$, then
    \begin{center}
    	$\begin{cases}
    		L_{k}(u_{n}-\delta^{*}w)\ge F_{k}(D^{2}u_{n})-F_{k}(D^{2}\omega)
    		=\lambda_{n}f(-u_{n},-v_{n})-(\lambda_{0}\delta^{*}\vert z\vert)
    		\ge (\overline{\lambda}\vert v_{n}\vert)-(\lambda_{0}\delta^{*}\vert z\vert)\ge 0\\
    		L_{k}(v_{n}-\delta^{**}z)\ge F_{k}(D^{2}v_{n})-F_{k}(D^{2}\tau)
    		=\lambda_{n}g(-u_{n},-v_{n})-(\lambda_{0}\delta^{**}\vert w\vert)
    		\ge (\overline{\lambda}\vert u_{n}\vert)-(\lambda_{0}\delta^{**}\vert w\vert)\ge 0
    	\end{cases}$
    \end{center}
since $\overline{\lambda}>\lambda_{1}$ and $0<\delta^{*}\vert w\vert\le \vert u_{n}\vert $, $0<\delta^{**}\vert z\vert\le \vert v_{n}\vert $ for all $x\in \Omega$. That implies, by strong maximum principle, $u_{n}(x)=\delta^{*}w$ and $v_{n}(x)=\delta^{*}z$ for all $x\in \Omega$. Therefore,
    \begin{center}
    	$\begin{cases}
    		S_{k}(D^{2}u_{n})=S_{k}(D^{2}\omega)\\ 
    		S_{k}(D^{2}v_{n})=S_{k}(D^{2}\tau)
    	\end{cases}$
    \end{center}  
or equivalently,
    \begin{center}
    	$\begin{cases}
    		\lambda_{1}^{k}\vert u_{n}\vert ^{k}=\lambda_{1}^{k}\vert \delta^{*}w\vert ^{k}=\lambda_{n}^{k}f^{k}(-u_{n},-v_{n})>\lambda_{1}^{k}\vert u_{n}\vert^{k}\\ 
    		\lambda_{1}^{k}\vert v_{n}\vert ^{k}=\lambda_{1}^{k}\vert \delta^{**}z\vert ^{k}=\lambda_{n}^{k}g^{k}(-u_{n},-v_{n})> \lambda_{1}^{k}\vert v_{n}\vert^{k}
    	\end{cases}$
    \end{center}
That is impossible, so only $\Vert u_{n}\Vert+\Vert v_{n}\Vert\rightarrow+\infty$ as $n\rightarrow+\infty$.

Let $\hat{g}(s,t),~\hat{h}(s,t)$ be such that 
    \begin{center}
    	$\begin{cases}
    		g^{k}(s,t)=g_{\infty}^{k}\vert t\vert^{k}+\hat{g}(s,t)\\
    		h^{k}(s,t)=h_{\infty}^{k}\vert s\vert^{k}+\hat{h}(s,t)
    	\end{cases}	$
    \end{center}
with $\underset{\vert s\vert +\vert t\vert \rightarrow+\infty}{lim}\frac{\hat{g}(s,t)}{(\vert s\vert +\vert t\vert )^{k}}=0,~\underset{\vert s\vert +\vert t\vert \rightarrow+\infty}{lim}\frac{\hat{h}(s,t)}{(\vert s\vert +\vert t\vert )^{k}}=0$. For $\Vert u_{n}\Vert+\Vert v_{n} \Vert \rightarrow+\infty$, we consider the following equation
    \begin{center}
    	$\begin{cases}
    		S_{k}(D^{2}u_{n})=\lambda_{n}^{k}(g_{\infty}^{k}\vert v_{n}\vert ^{k}+\hat{g}(-u_{n},-v_{n}))\qquad &in \quad\Omega\\ 
    		S_{k}(D^{2}v_{n})=\lambda_{n}^{k}(h_{\infty}^{k}\vert u_{n}\vert ^{k}+\hat{h}(-u_{n},-v_{n}))\qquad &in \quad\Omega\\ 
    		u_{n}=v_{n}=0\qquad&on\quad\partial\Omega
    	\end{cases}$
    \end{center}
Let $w_{n}=\frac{u_{n}}{\Vert u_{n}\Vert+\Vert v_{n}\Vert }$ and $z_{n}=\frac{v_{n}}{\Vert u_{n}\Vert+\Vert v_{n}\Vert }$. Then $(w_{n},z_{n})$ is the admissible solution to
    \begin{center}
    	$\begin{cases}
    		S_{k}(D^{2}w)=\lambda_{n}^{k}(g_{\infty}^{k}\vert z\vert ^{k}+\frac{\hat{g}(-u_{n},-v_{n})}{	(\Vert u_{n}\Vert+\Vert v_{n}\Vert)^{k}})\qquad &in \quad\Omega\\ 
    		S_{k}(D^{2}z)=\lambda_{n}^{k}(h_{\infty}^{k}\vert w\vert ^{k}+\frac{\hat{h}(-u_{n},-v_{n})}{	(\Vert u_{n}\Vert+\Vert v_{n}\Vert)^{k}})\qquad &in \quad\Omega\\ 
    		w=z=0\qquad&on\quad\partial\Omega
    	\end{cases}$
    \end{center}
By the same argument as in \cref{prop:SKEBifur2}, we have
    \begin{center}
    	$\frac{\hat{g}(-u_{n},-v_{n})}{	(\Vert u_{n}\Vert+\Vert v_{n}\Vert)^{k}}\rightarrow 0,\quad\frac{\hat{h}(-u_{n},-v_{n})}{	(\Vert u_{n}\Vert+\Vert v_{n}\Vert)^{k}}\rightarrow0$,
    \end{center}
as $n\rightarrow+\infty$. Since $w_{n}$ and $z_{n}$ are bounded, after taking a sub-sequence if necessary, we have that $w_{n}\rightarrow w$ and $z_{n}\rightarrow z$ in $C(\overline{\Omega})$ as $n\rightarrow+\infty$ and $(w,z)$ solve the equation 
    \begin{center}
    	$\begin{cases}
    		S_{k}(D^{2}w)=\widetilde{\lambda}^{k}g_{\infty}^{k}\vert z\vert ^{k}\qquad &in \quad\Omega\\ 
    		S_{k}(D^{2}z)=\widetilde{\lambda}^{k}h_{\infty}^{k}\vert w\vert ^{k}\qquad &in \quad\Omega\\ 
    		w=z=0\qquad&on\quad\partial\Omega
    	\end{cases}$
    \end{center}
where $\widetilde{\lambda}=\underset{n\rightarrow}{lim}~\lambda_{n}$, then by \cref{coro:SKSE}, we have $\widetilde{\lambda}=\frac{\lambda_{0}}{f_{\infty}}$. Therefore, $\mathcal{C}$ links $(\lambda_{0}/f_{0},(0,0))$ to $(\lambda_{0}/f_{\infty},\infty)$. So \eqref{eq:h} has at least one admissible solution $(u,v)\in E$ for every 
    \begin{center}
        $\lambda\in (min\{\frac{\lambda_{0}}{f_{0}},\frac{\lambda_{0}}{f_{\infty}}\},max\{\frac{\lambda_{0}}{f_{0}},\frac{\lambda_{0}}{f_{\infty}}\})$. 
    \end{center}
    
(2): Considering the proof of (1), we only need to show that $\mathcal{C}$ joints $(\lambda_{0}/f_{0},(0,0))$ to $(+\infty,+\infty)$. We first show that $\mathcal{C}$ is unbounded in the direction of $E$. Suppose on the contrary, that $\mathcal{C}$ is bounded in the direction of $E$. Then $\mathcal{C}$ is unbounded in the direction of $\lambda$. So there exists $(\lambda_{n},(u_{n},v_{n}))$ and $M>0$ such that $\lambda_{n}\rightarrow+\infty$ as $n\rightarrow+\infty$ and $\Vert u_{n}\Vert +\Vert v_{n}\Vert \le M$ for any $n\in \mathbb{N}$. The sign condition $g(s,t)>0$, $h(s,t)>0$ combined with $f_{0}\in (0,+\infty)$ and $\Vert u_{n}\Vert +\Vert v_{n}\Vert \le M$ implies that there exists a positive constant $\rho$ such that
	\begin{center}
		$\frac{g^{k}(-u_{n},-v_{n})}{\vert v_{n}\vert ^{k}}\ge \rho$ and $\frac{h^{k}(-u_{n},-v_{n})}{\vert u_{n}\vert ^{k}}\ge\rho$.
	\end{center}
So there exist a positive constant $\sigma\ge\rho $ such that 
	\begin{center}
		$g^{k}(-u_{n},-v_{n})/\vert v_{n}\vert^{k} \ge \sigma,\quad h^{k}(-u_{n},-v_{n})/\vert u_{n}\vert^{k} \ge \sigma~for ~any ~n\in \mathbb{N}$.
	\end{center}
So we have $\underset{n\rightarrow+\infty}{lim}\frac{\lambda_{n}^{k}g^{k}(-u_{n},-v_{n})}{\vert v_{n}\vert ^{k}}=\underset{n\rightarrow+\infty}{lim}\frac{\lambda_{n}^{k}h^{k}(-u_{n},-v_{n})}{\vert u_{n}\vert ^{k}}=+\infty$, by \cref{pro:priori2} if necessary, for n large we have $\Vert u_n\Vert +\Vert v_{n}\Vert\le C$ where $C>0$ is a constant and $\frac{\lambda_{n}^{k}g^{k}(-u_{n},-v_{n})}{\vert v_{n}\vert ^{k}}>\overline{\lambda}^{k}>\lambda_{1}^{k}$ and $\frac{\lambda_{n}^{k}h^{k}(-u_{n},-v_{n})}{\vert u_{n}\vert ^{k}}>\overline{\lambda}^{k}> \lambda_{1}^{k}$ for some constant $\overline{\lambda}> \lambda_{1}$ for $n$ large enough. Then, argument as in (1), we will get contradiction. So $\mathcal{C}$ is unbounded in the direction of $E$.

We claim that unique blow up point of $\mathcal{C}$ is $(+\infty,\infty)$.

In fact, otherwise, there must exist a blow up point with $\hat{\lambda}<\infty$ of $\mathcal{C}$. Then there exists a sequence $\left\{(\lambda_{n},(u_{n},v_{n}))\right\}$ such that $\underset{n\rightarrow+\infty}{lim}~\lambda_{n}=\hat{\lambda}$ and $\underset{n\rightarrow+\infty}{lim}(\Vert u_{n}\Vert +\Vert v\Vert )=+\infty$. Let $w_{n}=\frac{u_{n}}{\Vert u_{n}\Vert+\Vert v_{n}\Vert }$ and $z_{n}=\frac{v_{n}}{\Vert u_{n}\Vert+\Vert v_{n}\Vert }$. Then $(w_{n},z_{n})$ satisfy
    \begin{center}
    	$\begin{cases}
    		S_{k}(D^{2}w_{n})=\lambda_{n}^{k}\frac{g^{k}(-u_{n},-v_{n})}{	(\Vert u_{n}\Vert+\Vert v_{n}\Vert)^{k}}\qquad &in \quad\Omega\\ 
    		S_{k}(D^{2}z_{n})=\lambda_{n}^{k}\frac{h^{k}(-u_{n},-v_{n})}{	(\Vert u_{n}\Vert+\Vert v_{n}\Vert)^{k}}\qquad &in \quad\Omega\\ 
    		w_{n}=z_{n}=0\qquad&on\quad\partial\Omega
    	\end{cases}$.
    \end{center}
Similar to that of (1), we can show that $w_{n}\rightarrow w_{0}$ and $z_{n}\rightarrow z_{0}$ as $n\rightarrow +\infty$. Furthermore, $(w_{0},v_{0})$ is the solution of the following problem
    \begin{equation}\label{eq:CSK10}
    	\begin{cases}
    		S_{k}(D^{2}w)=0\qquad &in \quad\Omega\\ 
    		S_{k}(D^{2}z)=0\qquad &in \quad\Omega\\ 
    		w=z=0\qquad&on\quad\partial\Omega
    	\end{cases}.
    \end{equation}
Then, we have $w_{0}=z_{0}\equiv0$, which contradicts $\Vert w_{0}\Vert +\Vert z_{0}\Vert =1$.

(3): Argument as in (1), there exists an unbounded continuum $\mathcal{C}$ of admissible solutions of \cref{eq:CSK1} emanating from $(\lambda_{1}/ h_{0},(0,0))$ such that $\mathcal{C}\subset(\mathbb{R}\times E)\cup \left\{(\lambda_{1}/ h_{0},(0,0))\right\}$.

We claim that $\mathcal{C}$ is bounded in the direction of $\lambda$.

In fact, if we suppose that $\underset{n\rightarrow+\infty}{lim}\lambda_{n}=+\infty$. Since $f_{0},~f_{\infty}\in (0,+\infty)$, there exist a positive constant $\sigma$ such that 
	\begin{center}
		$g^{k}(-u_{n},-v_{n})/\vert v_{n}\vert^{k} \ge \sigma,\quad h^{k}(-u_{n},-v_{n})/\vert u_{n}\vert^{k} \ge \sigma~for ~any ~n\in \mathbb{N}$.
	\end{center}
So we have $\frac{\lambda_{n}^{k}g^{k}(-u_{n},-v_{n})}{\vert v_{n}\vert ^{k}}>\overline{\lambda}^{k}> \lambda_{1}^{k}$ and $\frac{\lambda_{n}^{k}h^{k}(-u_{n},-v_{n})}{\vert u_{n}\vert ^{k}}>\overline{\lambda}^{k}> \lambda_{1}^{k}$ for some constant $\overline{\lambda}> \lambda_{1}$ for $n$ large enough. Let $(w,z)$ satisfies 
	\begin{center}
		$\begin{cases}
			S_{k}(D^{2}w)=\lambda_{1}^{k}\vert z\vert ^{k}\qquad &in \quad\Omega\\ 
			S_{k}(D^{2}z)=\lambda_{1}^{k}\vert w\vert ^{k}\qquad &in \quad\Omega\\ 
			w=z=0\qquad&on\quad\partial\Omega
		\end{cases}$
	\end{center}
By scaling if necessary, for $n$ large enough, we can assume $u_{n}(x)<w(x)$ and $v_{n}(x)<z(x)$ for all $x\in \Omega$. Then as in (1), we would get a contradiction. In other words, $\mathcal{C}$ is unbounded in the direction of $E$. 

We claim that the unique blow up point of $\mathcal{C}$ is $(0,(0,0))$. 

Otherwise, there exists $\lambda_{*}>0$ such that $\left\{(\lambda_{n},(u_{n},v_{n}))\right\}\in \mathcal{C}$ satisfying $\underset{n\rightarrow+\infty}{lim}~\lambda_{n}=\lambda_{*}$ and $\underset{n\rightarrow+\infty}{lim}(\Vert u_{n}\Vert +\Vert v\Vert )=+\infty$. Since $f_{\infty}=g_{\infty}=\infty$ , thus for $n$ large enough we have $\frac{\lambda_{n}^{k}f^{k}(-u_{n},-v_{n})}{\vert v_{n}\vert ^{k}}>\overline{\lambda}^{k}> \lambda_{1}^{k}$ and $\frac{\lambda_{n}^{k}g^{k}(-u_{n},-v_{n})}{\vert u_{n}\vert ^{k}}>\overline{\lambda}^{k}> \lambda_{1}^{k}$ for some constant $\overline{\lambda}> \lambda_{1}$. Then, argument as in above, we would get a contradiction. Therefore, $\mathcal{C}$ links $(\lambda_{1}/ h_{0},(0,0))$ to $(0,\infty)$. So, \cref{eq:CSK1} has at least one admissible solution for every $\lambda\in (0,\frac{\lambda_{0}}{f_{0}})$.

(4): Let $\hat{g}$ and $\hat{h}$ be defined as in (1),then \cref{eq:CSK1} is equal to
    \begin{equation}\label{eq:CSK11}
    	\begin{cases}
    		S_{k}(D^{2}u)=\lambda^{k}(g_{\infty}^{k}\vert u\vert^{k}+\hat{g}(-u,-v))\qquad &in \quad\Omega\\ 
    		S_{k}(D^{2}v)=\lambda^{k}(h_{\infty}^{k}\vert v\vert^{k}+\hat{h}(-u,-v))\qquad &in \quad\Omega\\ 
    		w=z=0	\qquad &on \quad\partial\Omega	
    	\end{cases},
    \end{equation}
By \cref{prop:SKEBifur2}, there exists a continuum of admissible solutions of \cref{eq:CSK11} emanating from $(\lambda_{0}/h_{\infty},\infty)$ such that either it is unbounded in the direction of $\lambda$ or it meets $\left\{(\lambda,(0,0)):\lambda\in \mathbb{R}^{+}\right\}$.

We claim that $\mathcal{C}\cap (\mathbb{R}^{+}\times \left\{(0,0)\right\})=\emptyset$.

Otherwise, there exists a sequence $\left\{(\lambda_{n},(u_{n},v_{n}))\right\}$ such that $\underset{n\rightarrow+\infty}{lim}~\lambda_{n}=\mu $ with $\mu <+\infty$ and $\Vert  u_{n}\Vert +\Vert v_{n}\Vert\rightarrow 0 $ as $n\rightarrow+\infty$. Let $w_{n}=\frac{u_{n}}{\Vert u_{n}\Vert +\Vert v_{n}\Vert}$, $z_{n}=\frac{v_{n}}{\Vert u_{n}\Vert +\Vert v_{n}\Vert}$ and $(w_{n},z_{n})$ should be the solutions of the following problem
    \begin{center}
    	$\begin{cases}
    		S_{k}(D^{2}(w_{n}))=\lambda_{n}^{k}\frac{g^{k}(-u_{n},-v_{n})}{(\Vert u_{n}\Vert +\Vert v_{n}\Vert)^{k}}\qquad &in \quad\Omega\\ 
    		S_{k}(D^{2}(z_{n}))=\lambda_{n}^{k}\frac{h^{k}(-u_{n},-v_{n})}{(\Vert u_{n}\Vert +\Vert v_{n}\Vert)^{k}}\qquad &in \quad\Omega\\ 
    		w_{n}=z_{n}=0	\qquad &on \quad\partial\Omega
    	\end{cases}$.
    \end{center}
In view of $g_{0}=h_{0}=0$, like that of (2), we obtain that for some sub-sequence $(w_{n},z_{n})\rightarrow (w_{0},z_{0})$ in $C(\overline{\Omega})$ as $n\rightarrow+\infty$ and $(w_{0},z_{0})$ is generalized solution of \cref{eq:CSK10}. Then we obtain $w_{0}=z_{0}\equiv0$, which contradicts $\Vert w_{0}\Vert +\Vert z_{0}\Vert =1$. Therefore, $\mathcal{C}$ is unbounded in the direction of $\lambda$. 

In addition, we also have that $\mathcal{C}$ joins to $(+\infty,(0,0))$. Otherwise, there exists a positive constant $\rho $ and $(\lambda_{m},(u_{m},v_{m}))\in \mathcal{C}$ such that $\lambda_{m}\rightarrow+\infty$ as $m\rightarrow+\infty$ and $\Vert u_{m}\Vert+\Vert u_{m}\Vert\ge \rho $ for any $n\in \mathbb{N}$. Furthermore, since $g_{\infty}=h_{\infty}\in (0,+\infty)$, it follows that there exists a positive constant $\sigma$ such that
    \begin{center}
    	$\frac{g^{k}(-u_{m},-v_{m})}{\vert v_{m}\vert^{k}}\ge\frac{g^{k}(-u_{m},-v_{m})}{\Vert v_{m}\Vert^{k}} \ge \frac{g^{k}(-u_{m},-v_{m})}{(\Vert u_{m}\Vert+\Vert v_{m}\Vert)^{k}}\ge \sigma$
    \end{center} 
and 
    \begin{center}
    	$\frac{h^{k}(-u_{m},-v_{m})}{\vert u_{m}\vert^{k}}\ge\frac{h^{k}(-u_{m},-v_{m})}{\Vert u_{m}\Vert^{k}} \ge \frac{h^{k}(-u_{m},-v_{m})}{(\Vert u_{m}\Vert+\Vert v_{m}\Vert)^{k}}\ge \sigma$
    \end{center} 
for any $x\in \Omega$ and all $n\in \mathbb{N}$. So we have $\underset{m\rightarrow+\infty}{lim}\frac{\lambda_{m}^{k}g^{k}(-u_{m},-v_{m})}{\vert v_{m}\vert ^{k}}=\underset{m\rightarrow+\infty}{lim}\frac{\lambda_{m}^{k}h^{k}(-u_{m},-v_{m})}{\vert u_{m}\vert ^{k}}=+\infty$, by \cref{pro:priori2} if necessary, for n large we have $\Vert u_m\Vert +\Vert v_{m}\Vert\le C$ where $C>0$ is a constant and $\frac{\lambda_{m}^{k}g^{k}(-u_{m},-v_{m})}{\vert v_{m}\vert^{k}}\ge \overline{\lambda}^{k}>\lambda_{1}^{k}$ and $\frac{\lambda_{m}^{k}h^{k}(-u_{m},-v_{m})}{\vert u_{m}\vert^{k}}\ge \overline{\lambda}^{k}>\lambda_{1}^{k}$ for some constant $\overline{\lambda}>\lambda_{1}$. Let $(w,z)$ satisfy
    \begin{center}
    	$\begin{cases}
    		S_{k}(D^{2}w)=\lambda_{1}^{k}\vert z\vert^{k}\qquad &in \quad\Omega\\ 
    		S_{k}(D^{2}z)=\lambda_{1}^{k}\vert w\vert^{k}\qquad &in \quad\Omega\\ 
    		w=z=0	\qquad &on \quad\partial\Omega	
    	\end{cases}$.
    \end{center}
By scaling if necessary, for $m$ large enough, we can assume $u_{m}(x)<w(x)$ and $v_{m}(x)<z(x)$ for all $x\in \Omega$. Let $\delta^{*},\delta^{**}>0$ be the maximal such that $u_{m}-\delta^{*}w\le 0$, $v_{m}-\delta^{**}z\le 0$ in $\Omega$. Similar to (1), we will get contrary. So $\mathcal{C}$ joins to $(+\infty,(0,0))$. So \cref{eq:CSK1} has at least one nontrivial solution for every $\lambda\in (\lambda_{0}/f_{\infty},+\infty)$.

(5): Let
\begin{center} 
    $g^{k}_{n}(s,t)=
    \begin{cases}
		\vert\frac{1}{n} t\vert ^{k}+\overline{g}^{n}(s,t),\qquad&(\vert s\vert +\vert t\vert )\in [o,\frac{1}{n}]\\ 
	g^{k}(s,t),\qquad&(\vert s\vert +\vert t\vert )\in \left[\frac{2}{n},+\infty\right)
    \end{cases}$
\end{center}
and
\begin{center} 
    $h^{k}_{n}(s,t)=
    \begin{cases}
		\vert\frac{1}{n} s\vert ^{k}+\overline{h}^{n}(s,t),\qquad&(\vert s\vert +\vert t\vert )\in [o,\frac{1}{n}]\\ 
		h^{k}(s,t),\qquad&(\vert s\vert +\vert t\vert )\in \left[\frac{2}{n},+\infty\right)
    \end{cases}$
\end{center}
be continuous and $\underset{\vert s\vert +\vert t\vert\rightarrow 0}{lim}\frac{\overline{g}^{n}(s,t)}{\vert t\vert^{k}}=0$, $\underset{\vert s\vert +\vert t\vert\rightarrow 0}{lim}\frac{\overline{h}^{n}(s,t)}{\vert s\vert ^{k}}=0$. Then we consider the following problem
\begin{equation}\label{eq:CSK121}
	\begin{cases}
		S_{k}(D^{2}u)=\lambda^{k}g^{k}_{n}(-u,-v)\qquad&in \quad\Omega \\ 
		S_{k}(D^{2}v)=\lambda^{k}h^{k}_{n}(-u,-v)\qquad&in \quad\Omega \\ 
		u=v=0\qquad&on\quad\partial\Omega
	\end{cases}.
\end{equation} 
It is easy to show that $\underset{n\rightarrow+\infty}{lim}g_{n}(s,t)=g(s,t)$, $\underset{n\rightarrow+\infty}{lim}h_{n}(s,t)=h(s,t)$, $g_{n,\infty}=h_{n,\infty}=f_{\infty}=+\infty$ and $g_{n,0}=h_{n,0}=\frac{1}{n}$. Then by (3), there exists a continuum $C^n$ of admissible solutions of \cref{eq:CSK121} emanating from $(n\lambda_1,(0,0))$. Taking $z^{*}=(+\infty ,(0,0))$, we have that $z^{*}\in \underset{n\rightarrow+\infty}{lim}\mathcal{C}^{n}$. The compactness of $H$ implies that $(\bigcup_{n=1}^{+\infty}\mathcal{C}^{n})\cap B_{R}$ is pre-compact, where $B_{R}=\left\{z\in \mathbb{R}\times E:\Vert z\Vert_{\mathbb{R}\times E}<R\right\}$ for any $R>0$. Using lemma 2.5 in \cite{dai_two_2016}, we obtain that $\mathcal{C}:=\underset{n\rightarrow+\infty}{limsup}~\mathcal{C}^{n}$ is an unbounded connected set such that $z^{*}\in \mathcal{C}$ and $C$ is the solution set of \cref{eq:SKS}.

We claim that $(0,(0,0)) $ is the unique blow up point of $C$. For if not, there exist $\lambda_{*}>0$ and $(\lambda_n,(u_n,v_n))\in C$ satisfying that $\underset{n\rightarrow+\infty}{\lambda_n}=\lambda_{*}$ and $\underset{n\rightarrow+\infty}{\Vert u_n\Vert +\Vert v_n\Vert}=+\infty$. But since $g_\infty=h_\infty=+\infty$, then by \cref{pro:priori2}, there exists a constant $M>0$ such that $\Vert u_n\Vert +\Vert v_n\Vert\le M$ for n large enough, contrary.\\
(6): Let
\begin{center} 
    $g^{k}_{n}(s,t)=
    \begin{cases}
		\vert n t\vert ^{k}+\overline{g}^{n}(s,t),\qquad&(\vert s\vert +\vert t\vert )\in [o,\frac{1}{n}]\\ 
	g^{k}(s,t),\qquad&(\vert s\vert +\vert t\vert )\in \left[\frac{2}{n},+\infty\right)
    \end{cases}$
\end{center}
and
\begin{center} 
    $h^{k}_{n}(s,t)=
    \begin{cases}
		\vert n s\vert ^{k}+\overline{h}^{n}(s,t),\qquad&(\vert s\vert +\vert t\vert )\in [o,\frac{1}{n}]\\ 
		h^{k}(s,t),\qquad&(\vert s\vert +\vert t\vert )\in \left[\frac{2}{n},+\infty\right)
    \end{cases}$
\end{center}
be continuous and $\underset{\vert s\vert +\vert t\vert\rightarrow 0}{lim}\frac{\overline{g}^{n}(s,t)}{\vert t\vert^{k}}=0$, $\underset{\vert s\vert +\vert t\vert\rightarrow 0}{lim}\frac{\overline{h}^{n}(s,t)}{\vert s\vert ^{k}}=0$. Then we consider the following problem
\begin{equation}\label{eq:CSK12}
	\begin{cases}
		S_{k}(D^{2}u)=\lambda^{k}g^{k}_{n}(-u,-v)\qquad&in \quad\Omega \\ 
		S_{k}(D^{2}v)=\lambda^{k}h^{k}_{n}(-u,-v)\qquad&in \quad\Omega \\ 
		u=v=0\qquad&on\quad\partial\Omega
	\end{cases}.
\end{equation} 
It is easy to show that $\underset{n\rightarrow+\infty}{lim}g_{n}(s,t)=g(s,t)$, $\underset{n\rightarrow+\infty}{lim}h_{n}(s,t)=h(s,t)$, $g_{n,\infty}=h_{n,\infty}=f_{\infty}\in (0,\infty)$ and $g_{n,0}=h_{n,0}=n$. By (1), we get that there exists an unbounded continuum $\mathcal{C}^{n}$ of the set of admissible solution of  \cref{eq:CSK12} emanating from $(\lambda_{1}/n,(0,0))$ such that $\mathcal{C}^{n}\subset \mathbb{R}\times E\cup (\lambda_{1}/n,(0,0))$ and $\mathcal{C}^{n}$ links to $(\frac{\lambda_{1}}{f_{\infty}},\infty)$. Taking $z^{*}=(0,(0,0))$, we have that $z^{*}\in \underset{n\rightarrow+\infty}{lim}\mathcal{C}^{n}$. The compactness of $H$ implies that $(\bigcup_{n=1}^{+\infty}\mathcal{C}^{n})\cap B_{R}$ is pre-compact, where $B_{R}=\left\{z\in \mathbb{R}\times E:\Vert z\Vert_{\mathbb{R}\times E}<R\right\}$ for any $R>0$. Using lemma 2.5 in \cite{dai_two_2016}, we obtain that $\mathcal{C}:=\underset{n\rightarrow+\infty}{limsup}~\mathcal{C}^{n}$ is an unbounded connected set such that $z^{*}\in \mathcal{C}$. In other words, $\mathcal{C}$ links $(0,(0,0))$ to $(\frac{\lambda_{1}}{f_{\infty}},\infty)$.

(7):Define 
	\begin{center}
		$g^{k}_{n}(s,t)=
		\begin{cases}
			n^{k}\vert t\vert ^{k}+\widetilde{g}^{n}(s,t),\qquad&(\vert s\vert +\vert t\vert )\in [0,\frac{1}{n}]\\ 	
			g^{k}(s,t),\qquad&(\vert s\vert +\vert t\vert)\in \left[\frac{2}{n},+\infty\right)
		\end{cases}$
	\end{center}
	and
	\begin{center}
		$	h^{k}_{n}(s,t)=
		\begin{cases}
			n^{k}\vert s\vert ^{k}+\widetilde{h}^{n}(s,t),\qquad&(\vert s\vert +\vert t\vert )\in [0,\frac{1}{n}]\\ 	
			h^{k}(s,t),\qquad&(\vert s\vert +\vert t\vert )\in \left[\frac{2}{n},+\infty\right)
		\end{cases}$
	\end{center}
be continuous and $\underset{\vert s\vert +\vert t\vert\rightarrow 0}{lim}\frac{\widetilde{g}^{n}(s,t)}{\vert t\vert^{k}}=0$, $\underset{\vert s\vert +\vert t\vert\rightarrow0}{lim}\frac{\widetilde{h}^{n}(s,t)}{\vert s\vert ^{k}}=0$. Then we consider the following problem
	\begin{equation}\label{eq:CSK13}
		\begin{cases}
			S_{k}(D^{2}u)=\lambda^{k}g^{k}_{n}(-u,-v)\qquad&in \quad\Omega \\ 
			S_{k}(D^{2}v)=\lambda^{k}h^{k}_{n}(-u,-v)\qquad&in \quad\Omega \\ 
			u=v=0\qquad&on\quad\partial\Omega
		\end{cases}
	\end{equation} 
It is easy to show that $\underset{n\rightarrow+\infty}{lim}g_{n}(s,t)=g(s,t)$, $\underset{n\rightarrow+\infty}{lim}h_{n}(s,t)=h(s,t)$, $g_{n,\infty}=h_{n,\infty}=0$ and $g_{n,0}=h_{n,0}=n$. By \cref{prop:SKEBifur1}, we get that there exists an unbounded continuum $\mathcal{C}^{n}$ of the set of admissible solution of \cref{eq:CSK13} emanating $(\lambda_{1}/n,(0,0))$ such that $\mathcal{C}^{n}\subset \mathbb{R}\times E\cup (\lambda_{1}/n,(0,0))$. Taking $z^{*}=(0,(0,0))$, we have that $z^{*}\in \underset{n\rightarrow+\infty}{lim}\mathcal{C}^{n}$. The compactness of $H$ implies that $(\bigcup_{n=1}^{+\infty}\mathcal{C}^{n})\cap B_{R}$ is precompact, where $B_{R}=\left\{z\in \mathbb{R}\times E:\Vert z\Vert_{\mathbb{R}\times E}<R\right\}$ for any $R>0$. Using Lemma 2.5 in \cite{dai_two_2016}, we obtain that $\mathcal{C}:=\underset{n\rightarrow+\infty}{limsup}~\mathcal{C}^{n}$ is an unbounded connected set such that $z^{*}\in \mathcal{C}$.

Since $\mathcal{C}$ is unbounded, we have that there exists $(\lambda_{n},(u_{n},v_{n}))\in \mathcal{C}$ such that $\lambda_{n}+\Vert u_{n}\Vert +\Vert v_{n}\Vert \rightarrow+\infty$ as $n\rightarrow+\infty$.

We claim that $\lambda_{n}\rightarrow+\infty$ and $\Vert u_{n}\Vert +\Vert v_{n}\Vert \rightarrow+\infty$ as $n\rightarrow+\infty$. 

Otherwise, $\lambda_{n}\rightarrow+\infty$ or $\Vert u_{n}\Vert +\Vert v_{n}\Vert \rightarrow+\infty$ as $n\rightarrow+\infty$. If $\lambda_{n}\rightarrow+\infty$ as $n\rightarrow+\infty$ and $\Vert u_{n}\Vert +\Vert v_{n}\Vert\le M$ for some positive constant $M$. 
    
Then for $n$ large enough, there exist a positive constant $\sigma$ such that 
    \begin{center}
    	$g^{k}(-u_{n},-v_{n})/\vert v_{n}\vert^{k} \ge \sigma,\quad h^{k}(-u_{n},-v_{n})/\vert u_{n}\vert^{k} \ge \sigma~for ~any ~n\in \mathbb{N}$.
    \end{center}
So we have $\frac{\lambda_{n}^{k}g^{k}(-u_{n},-v_{n})}{\vert v_{n}\vert ^{k}}>\overline{\lambda}^{k}>\lambda_{1}^{k}$ and $\frac{\lambda_{n}^{k}h^{k}(-u_{n},-v_{n})}{\vert u_{n}\vert ^{k}}>\overline{\lambda}^{k}>\lambda_{1}^{k}$ for some constant $\overline{\lambda}>\lambda_{1}$. Argue as in (1), we will get contradiction. Furthermore, if $\lambda_{n}$ is bounded and $\Vert u_{n}\Vert +\Vert v_{n}\Vert \rightarrow+\infty$ as $n\rightarrow+\infty$. Let $w_{n}=\frac{u_{n}}{\Vert u_{n}\Vert+\Vert v_{n}\Vert}$ and $z_{n}=\frac{v_{n}}{\Vert u_{n}\Vert+\Vert v_{n}\Vert }$ and then $(w_{n},z_{n})$ satisfy
    \begin{center}
    	$\begin{cases}
    		S_{k}(D^{2}w_{n})=\lambda_{n}^{k}\frac{g^{k}(-u_{n},-v_{n})}{	(\Vert u_{n}\Vert+\Vert v_{n}\Vert)^{k}}\qquad &in \quad\Omega\\ 
    		S_{k}(D^{2}z_{n})=\lambda_{n}^{k}\frac{h^{k}(-u_{n},-v_{n})}{	(\Vert u_{n}\Vert+\Vert v_{n}\Vert)^{k}}\qquad &in \quad\Omega\\ 
    		w_{n}=z_{n}=0\qquad&on\quad\partial\Omega
    	\end{cases}$.
    \end{center}
Argue as in (2), we will also get contradiction. So $\mathcal{C}$ joins to $(+\infty,\infty)$.
\end{proof}

\subsection{Proof of \cref{thm:SKS2,thm:SKS3}}
Now, we begin to prove \cref{thm:SKS2}.

\newenvironment{prf3}{{\noindent\bf Proof of \cref{thm:SKS2}.}}{\hfill $\square$\par}
\begin{prf3}

In view of $g_{0}=h_{0}=+\infty$, from the argument of (6) in the proof of \cref{thm:SKS1}, we know that there is an unbounded continuum $\mathcal{C}$ of admissible solution to \cref{eq:SKS} emanating from $(0,(0,0))\in \mathcal{C}$. Since $g_{0}=h_{0}=+\infty,~g_{\infty}=h_{\infty}=+\infty$ and the sign condition $g(s,t),~h(s,t)>0$ for any $s,t>0$. If $\mathcal{C}$ is unbounded in the direction of $\lambda$, then argument as (1), we can get contradiction. It follows that $\mathcal{C}$ is unbounded in the direction of $E$.

We claim that $(0,(0,0))$ is the unique blow up point of $\mathcal{C}$. 

Otherwise, there must exist a blow up point $(\widetilde{\lambda},(0,0))$ with $\widetilde{\lambda}\in (0,+\infty )$ of $\mathcal{C}$. And then there exist $\left\{(\lambda_{n},(u_{n},v_{n}))\right\}$ such that $\underset{n\rightarrow+\infty}{lim}\lambda_{n}=\widetilde{\lambda}$ and $\underset{n\rightarrow+\infty}{lim}(\Vert u_{n}\Vert +\Vert v_{n}\Vert)=+\infty$. Since $g_{\infty}=h_{\infty}=+\infty$, we have for any $M>0$, for $n$ large enough,
    \begin{center}
    	$\frac{g^{k}(-u_{n},-v_{n})}{\vert v_{n}\vert^{k} }\ge M$ and $\frac{h^{k}(-u_{n},-v_{n})}{\vert u_{n}\vert ^{k} }\ge M$.
    \end{center} 
Then, we get $\underset{n\rightarrow+\infty}{lim}\frac{\lambda_{n}^{k}g^{k}(-u_{n},-v_{n})}{\vert v_{n}\vert ^{k}}=\underset{n\rightarrow+\infty}{lim}\frac{\lambda_{n}^{k}h^{k}(-u_{n},-v_{n})}{\vert u_{n}\vert ^{k}}=+\infty$, by \cref{pro:priori2} we have $\Vert u_n\Vert +\Vert v_{n}\Vert\le C$ for a constant $C>0$ which is contradiction with $\underset{n\rightarrow+\infty}{lim}(\Vert u_{n}\Vert +\Vert v_{n}\Vert)=+\infty$. Therefore, $(0,(0,0))$ is the unique blow-up point.

Let $\lambda^{*}=sup_{(\lambda,(u,v))\in \mathcal{C}}\left\{\lambda\right\}$. It is clearly that $\lambda^{*}\in (0,+\infty)$ so \cref{eq:SKS} has at least two admissible solutions for $\lambda>0$ small enough. Obviously, there exists at least one admissible solution at $\lambda=\lambda^{*}$. We divide the rest of proof into two steps. 

{\bf Step 1 }. For $\lambda>\lambda^{*}$, \cref{eq:CSK1} has no solution.

Assume hat there exists an admissible solution $(\overline{u},\overline{v})$ of \cref{eq:CSK1} for some $\overline{\lambda}>\lambda^{*}$. By \cref{pro:priori2}, we have $\Vert \overline{u}\Vert+\Vert \overline{v}\Vert\le C$ for some positive constant $C$ that only depend on $f,k~and ~\Omega$. Let $\mathcal{C}'=\left\{(\lambda,(u,v))\in \mathcal{C}:u\ge\overline{u}~and~v\ge\overline{v}\right\}$.

We $claim~\mathcal{C}'=\mathcal{C}$.

It is sufficient  to show that $\mathcal{C}'$ is not empty, open and closed relative to $\mathcal{C}$. It is clearly that $\mathcal{C}$ is nonempty since $(0,(0,0))\in \mathcal{C}$. The closeness of $\mathcal{C}'$ is obvious by the closeness of $\mathcal{C}$ and the definition of $\mathcal{C}'$. So we only need to show that $\mathcal{C}'$ is open relative to $\mathcal{C}$. For any $(\lambda_{0},(u_{0},v_{0}))\in \mathcal{C}'$, we have that $u_{0}>\overline{u}~in ~\Omega$ and $v_{0}>\overline{v}~in ~\Omega$. For if not, there exists $x_{0}\in \Omega$ such that $u_{0}(x_{0})=\overline{u}(x_{0})$ or $v_{0}(x_{0})=\overline{v}(x_{0})$. Without loss of generality,we can assume that$u_{0}(x_{0})=\overline{u}(x_{0})$. Thus $x_{0}$ is a locally minimal point of $u_{0}-\overline{u}$, then $D_{ij}[u_{0}-\overline{u}](x_{0})$ is positive semi-definite. Since $u_{0},~\overline{u}\in \Phi^{k}(\Omega)$, then by \cite{bhattacharya_maximum_2021},
    \begin{center}
    	$S_{k}(D^{2}u_{0}(x_{0}))\ge S_{k}(D^{2})\overline{u}(x_{0})$, i.e. $\lambda_{0}^{k}g^{k}(-u_{0}(x_{0}),-v_{0}(x_{0}))\ge \overline{\lambda}^{k}g^{k}(-\overline{u}(x_{0}),-\overline{v}(x_{0}))$
    \end{center}
But since $\overline{\lambda}>\lambda^{*}\ge \lambda_{0}$, $u_{0}(x_{0})=\overline{u}(x_{0})$, $v_{0}(x_{0})\ge \overline{v}(x_{0})$ and $g$ is non-decreasing of $t$. The conclusion is impossible. Therefore, $u_{0}>\overline{u}~in ~\Omega$ and $v_{0}>\overline{v}~in ~\Omega$.

For any $\varepsilon>0$ small enough, we choose an open neighborhood $B\subset \mathbb{R}\times E$ of $(\lambda_{0},(u_{0},v_{0}))$ such that $\Vert u-u_{0}\Vert +\Vert v-v_{0}\Vert<\varepsilon$ for any $(\lambda,(u,v))\in B\cap \mathcal{C}$. Then we have $u\ge \overline{u}$ ~and ~$v\ge \overline{v}$ in $\Omega_{\varepsilon}:=\left\{x:u_{0}\ge \overline{u}+\varepsilon ~and~ v_{0}\ge \overline{v}+\varepsilon\right\}$.	

We claim that there exists $\varepsilon$ small enough such that $u\ge \overline{u}$ ~and ~$v\ge \overline{v}$ in $\Omega':=\Omega\backslash\Omega_{\varepsilon}$.

In fact, for $\varepsilon$ small enough, let $L_{u}=\Sigma_{i,j=1}^{n}F_{ij}(u)D_{ij}$, then $L_{u}$ is linear elliptic operator if $u\in \Phi^{k}(\Omega)$. By \cite{caffarelli_dirichlet_1985}
\begin{center}
	$L_{u}(v-u)=\Sigma_{i,j=1}^{n}F_{ij}(u)D_{ij}(v-u)\ge F_{k}(v)-F_{k}(u)$
\end{center}
for any $u,~v\in \Phi^{k}(\Omega)$. Thus, by linearity $L_{u}(u-v)\le F_{k}(u)-F_{k}(v)$. Let $w_{1}(x)=u(x)-\overline{u}(x)$, $w_{2}(x)=v(x)-\overline{v}(x)$, then we have
\begin{center}
	$\begin{cases}
		L_{u}(u-\overline{u})-c_{11}(x)w_{1}(x)-c_{12}w_{2}(x)\le 0,\qquad&x\in\Omega^{'},\\
		L_{v}(v-\overline{v})-c_{21}(x)w_{1}(x)-c_{22}w_{2}(x)\le 0,\qquad&x\in\Omega^{'},\\
		w_{1}(x)\ge 0,w_{2}(x)\ge 0,\qquad&x\in \partial\Omega^{'},
	\end{cases}$
\end{center}
where
\begin{center}
	$c_{11}(x)=\frac{\lambda g(-u(x),-v(x))-\lambda g(-\overline{u}(x),-v(x))}{u(x)-\overline{u}(x)}$, $c_{12}(x)=\frac{ \lambda g(-\overline{u}(x),-v(x))-\overline{\lambda}g(-\overline{u}(x),-\overline{v}(x))}{v(x)-\overline{v}(x)}$
\end{center}
and
\begin{center}
	$c_{21}(x)=\frac{\lambda h(-u(x),-v(x))-\overline{\lambda}h(-\overline{u}(x),-v(x))}{u(x)-\overline{u}(x)}$, $c_{22}(x)=\frac{\overline{\lambda}h(-\overline{u}(x),-v(x))-\overline{\lambda}h(-\overline{u}(x),-\overline{v}(x))}{v(x)-\overline{v}(x)}$
\end{center}
Then by \cite{de_figueredo_monotonicity_1994}, for $\varepsilon $ small enough, we have $w_{1}(x)=u(x)-\overline{u}(x)\ge 0$ and $w_{2}(x)=v(x)-\overline{v}(x)\ge 0$ in $\Omega^{'}$. So $\mathcal{C}'$ is open relative to $\mathcal{C}$. Therefore, we get $\mathcal{C}'=\mathcal{C}$.

Since $\mathcal{C}'=\mathcal{C}$ is unbounded in the direction of X, there exists $(\lambda_{n},(u_{n},v_{n}))\in \mathcal{C}'$ such that $\Vert u_{n}\Vert+\Vert v_{n}\Vert\rightarrow+\infty$ as $n\rightarrow+\infty$. So we have $\Vert \overline{u}\Vert+\Vert \overline{v}\Vert=+\infty$, which is a contradiction.

{\bf Step 2 }. For $\lambda<\lambda_{*}$, \cref{eq:SKS} has at least two admissible solutions.

Let$\widetilde{\lambda}=sup\left\{\lambda>0:\exists (u_{1},v_{1})\not\equiv (u_{2},v_{2})~s.t.~ (\lambda,(u_{1},v_{1})),(\lambda,(u_{2},v_{2}))\in \mathcal{C}\right\}$. Suppose on the contrary, that $\widetilde{\lambda}<\lambda^{*}$. Take $\hat{\lambda}\in \left[\widetilde{\lambda},\lambda^{*}\right)$ such that there exists a unique $(\hat{u},\hat{v})$ such that $(\hat{\lambda},(\hat{u},\hat{v}))\in \mathcal{C}$. Define
    \begin{center}
    	$\mathcal{C}''=\left\{(\lambda,(u,v))\in \mathcal{C}:\lambda\le \hat{\lambda}\right\}$.
    \end{center}
Clearly, $\mathcal{C}''$ is connected and closed. Then let $\overline{\mathcal{C}}=\left\{(\lambda,(u,v))\in \mathcal{C}'':u\ge u''~or~v\ge v''\right\}$, where $(\lambda'',(u'',v''))$ is a solution of \cref{eq:CSK1} and $\lambda''>\hat{\lambda}$. So we can repeat the argument of Step 1 to show the \cref{eq:SKS} have no solution for $\lambda>\hat{\lambda}$, which is a contradiction.
\end{prf3}

\vskip 0.2in

Next, we begin to prove \cref{thm:SKS3}.

\newenvironment{prf4}{{\noindent\bf Proof of \cref{thm:SKS3}.}}{\hfill $\square$\par}
\begin{prf4}
Let 
    \begin{center}
    	$g^{k}_{n}(s,t)=	
    	\begin{cases}
    		\frac{1}{n}	\vert t\vert ^{k}+\overline{g}^{n}(s,t) \qquad&s\in [0,\frac{1}{n}]\\ \\
    		g^{k}(s,t)\qquad&s\in \left[\frac{2}{n},+\infty\right)
    	\end{cases}$
    \end{center}
    and
    \begin{center}
    	$h^{k}_{n}(s,t)=	
    	\begin{cases}
    		\frac{1}{n}	\vert s\vert ^{k}+\overline{h}^{n}(s,t) \qquad&s\in [0,\frac{1}{n}]\\ \\
    		h^{k}(s,t)\qquad&s\in \left[\frac{2}{n},+\infty\right)
    	\end{cases}$
    \end{center}
where $\underset{\vert s\vert +\vert t\vert \rightarrow0}{lim}\frac{\overline{g}^{k}_{n}(s,t)}{\vert t\vert^{k}}=0$ and $\underset{\vert s\vert +\vert t\vert \rightarrow0}{lim}\frac{\overline{h}^{k}_{n}(s,t)}{\vert s\vert ^{k}}=0$. We have $\underset{n\rightarrow+\infty}{lim}g_{n}(s,t)=g(s,t)$, $\underset{n\rightarrow+\infty}{lim}h_{n}(s,t)=h(s,t)$  and $g^{n,0}=h^{n,0}=\frac{1}{n}$, $g_{n,\infty}=h_{n,\infty}=0$. Then we consider the following equation 
    \begin{equation}\label{eq:CSK16}
    	\begin{cases}
    		S_{k}(D^{2}u)=\lambda^{k}g^{k}_{n}(-u,-v)\qquad&in \quad\Omega \\ 
    		S_{k}(D^{2}v)=\lambda^{k}h^{n}(u,v)\qquad&in \quad\Omega \\ 
    		u=v=0\qquad&on\quad\partial\Omega
    	\end{cases}.
    \end{equation}
By (2) we get that continuum  $\mathcal{C}_{n}$ of admissible solutions emanating $(n\lambda_{1},(0,0))$ and is unbounded. 

We claim that if $\left\{\lambda_{n},(u_{n},v_{n})\right\} \in \mathcal{C}_{n}$ satisfy $\lambda_{n}\rightarrow+\infty$ and $\vert u_{n}\vert +\vert v_{n}\vert \nrightarrow0$, then $\Vert u_{n}\Vert+\Vert v_{n}\Vert \rightarrow+\infty$. For if not, suppose that $0<\Vert u_{n}\Vert+\Vert v_{n}\Vert\le M$ for some constant $M>0$. Then, we have that there exists a positive constant $\sigma$ such that 
    \begin{center}
    	$\frac{g^{k}(-u_{n},-v_{n})}{\vert v_{n}\vert ^{k}}\ge \frac{g^{k}(-u_{n},-v_{n})}{\Vert v_{n}\Vert^{k} }\ge \frac{g^{k}(-u_{n},-v_{n})}{(\Vert u_{n}\Vert +\Vert v_{n}\Vert)^{k}}\ge \sigma$
    \end{center}
    and
    \begin{center}
    	$\frac{h^{k}(-u_{n},-v_{n})}{\vert v_{n}\vert ^{k}}\ge \frac{h^{k}(-u_{n},-v_{n})}{\Vert v_{n}\Vert^{k} }\ge \frac{h^{k}(-u_{n},-v_{n})}{(\Vert u_{n}\Vert +\Vert v_{n}\Vert)^{k}}\ge \sigma$
    \end{center}
for any $x\in \Omega$ and all $n\in \mathbb{N}$. So we have $\frac{\lambda_{n}^{k}f^{k}(-u_{n},-v_{n})}{\vert v_{n}\vert ^{k}}> \lambda_{1}^{k}$ and $\frac{\lambda_{n}^{k}g^{k}(-u_{n},-v_{n})}{\vert u_{n}\vert ^{k}}>\lambda_{1}^{k}$ for $n$ large enough. From the argument as (1), we will get contradiction. So, $\mathcal{C}_{n}$ link $(n\lambda_{0},(0,0))$ to $(+\infty,\infty )$ in $M$. Now, taking $z^{1}_{n}=(n\lambda_{0},(0,0))$, $z^{2}_{n}=(+\infty,+\infty)$. The compactness of $H$ in implies that $(\bigcup_{n=1}^{+\infty} \mathcal{C}_{n})\cap B_{R}$ is precompact. So by Lemma 4.11 in \cite{dai_eigenvalue_2015} there exists unbounded component $\mathcal{C}$ such that $(+\infty,(0,0))\in \mathcal{C}$ and $(+\infty,\infty)\in \mathcal{C}$.

Let $\lambda_{*}=\underset{(\lambda,(u,v))\in \mathcal{C}}{inf}\left\{\lambda\right\}$. Obviously, there exists at least one nontrivial admissible solution at $\lambda=\lambda_{*}$. We divide the rest of the proof into two steps.

{\bf Step 1 }. When $\lambda<\lambda_{*}$ (1.1) has no solution.

On the contrary, suppose that there exists $\underline{\lambda}$ such that $(\underline{u},\underline{v})$ is nontrivial admissible solution of \cref{eq:CSK1}. By \cref{prop:priori}, we have $\Vert \underline{u}\Vert+\Vert \underline{v}\Vert\le C$ for some positive constant $C$ that only depend on $g,~h,~k~and ~\Omega$. Let $\mathcal{C}'=\left\{(\lambda,(u,v))\in \mathcal{C}:u\le\underline{u}~and~v\le\underline{v}\right\}$.

We $claim~\mathcal{C}'=\mathcal{C}$.

It is sufficient to show that $\mathcal{C}'$ is not empty, open and closed relative to $\mathcal{C}$. It is clearly that $\mathcal{C}$ is nonempty since $g,~h$ are coercive and $(+\infty,\infty)\in \mathcal{C}$. In fact, there exists $(\lambda_{n},(u_{n},v_{n}))$ such that $\lambda_{n}\rightarrow+\infty$ and $\Vert u_{n}\Vert +\Vert v_{n}\Vert \rightarrow+\infty$ as $n\rightarrow+\infty$. Then for $n$ large enough, we have
    \begin{center}
    	$\lambda_{n}^{k}g^{k}(-u_{n},-v_{n})>\underline{\lambda}^{k}g^{k}(-\underline{u},-\underline{v})$,
    \end{center}
    and
    \begin{center}
    	$\lambda_{n}^{k}h^{k}(-u_{n},-v_{n})>\underline{\lambda}^{k}h^{k}(-\underline{u},-\underline{v})$,
    \end{center}
then by comparison principle we have for sufficient large $n$, $u_{n}\le\underline{u}$ and $v_{n}\le\underline{v}$. The closeness of $\mathcal{C}'$ is obvious by the closeness of $\mathcal{C}$ and the definition of $\mathcal{C}'$. So we only need to show that $\mathcal{C}'$ is open relative to $\mathcal{C}$.
For any $(\lambda_{0},(u_{0},v_{0}))\in \mathcal{C}'$, we have that $u_{0}<\underline{u}~in ~\Omega$ and $v_{0}<\underline{v}~in ~\Omega$. For if not, There exists $x_{0}\in \Omega$ such that $u_{0}(x_{0})=\underline{u}(x_{0})$ or $v_{0}(x_{0})=\underline{v}(x_{0})$. Without loss of generality, we can assume that $u_{0}(x_{0})=\underline{u}(x_{0})$, then we have $x_{0}$ is a locally minimum of $\underline{u}-u_{0}$, then $D_{ij}[\underline{u}-u_{0}](x_{0})$ is positive semidefinite. Since $u_{0},~\underline{u}\in \Phi^{k}(\Omega)$, then by \cite{bhattacharya_maximum_2021}, 
    \begin{center}
        $S_{k}(D^{2}\underline{u}(x_{0}))\ge S_{k}(D^{2})u_{0}(x_{0})$, i.e. $\underline{\lambda}^{k}g^{k}(-\underline{u}(x_{0}),-\underline{v}(x_{0}))\ge \lambda_{0}^{k}g^{k}(-u_{0}(x_{0}),-v_{0}(x_{0}))$
    \end{center}
But since $\underline{\lambda}<\lambda^{*}\le \lambda_{0}$, $u_{0}(x_{0})=\underline{u}(x_{0})$,$v_{0}(x_{0})\le \underline{v}(x_{0})$ and $f$ is non-decreasing of $t$. The conclusion is impossible. Therefore, $u_{0}<\underline{u}~in ~\Omega$ and $v_{0}<\underline{v}~in ~\Omega$.

For any $\varepsilon>0$, we choose an open neighborhood $B\subset \mathbb{R}\times E$ of $(\lambda_{0},(u_{0},v_{0}))$ such that $\Vert u-u_{0}\Vert+\Vert v-v_{0}\Vert <\varepsilon$. For any $(\lambda,(u,v))\in B\cap \mathcal{C}$, we have $u\le \underline{u} $ and $v\le \underline{v}$ in $\Omega_{\varepsilon}=\{x\in \Omega:u_{0}(x)\le \overline{u}(x)-\varepsilon~and~v_{0}(x)\le \overline{v}(x)-\varepsilon\}$. 

We claim that there exists $\varepsilon$ small enough such that $u\le \underline{u} $ and $v\le \underline{v}$ in $\Omega':=\Omega\backslash\Omega_{\varepsilon}$.

In fact, let $L_{u}=\Sigma_{i,j=1}^{n}F_{ij}(u)D_{ij}$,then $L_{u}$ is linear elliptic operator for $u\in \Phi^{k}(\Omega)$. By \cite{caffarelli_dirichlet_1985},
    \begin{center}
	   $L_{u}(v-u)=\Sigma_{i,j=1}^{n}F_{ij}(u)D_{ij}(v-u)\ge F_{k}(v)-F_{k}(u)$
    \end{center}
for any $u,~v\in \Phi^{k}(\Omega)$. Hence, by linearity, $L_{u}(u-v)\le F_{k}(u)-F_{k}(v)$. Let $w_{1}(x)=\underline{u}(x)-u(x)$ and $w_{2}(x)=\underline{v}(x)-v(x)$, then we have
\begin{center}
	$\begin{cases}
		L_{\underline{u}}(\underline{u}(x)-u(x))-c_{11}(x)w_{1}(x)-c_{12}w_{2}(x)\le 0,\qquad&x\in\Omega^{'},\\
		L_{\underline{v}}(\underline{v}(x)-v(x))-c_{21}(x)w_{1}(x)-c_{22}w_{2}(x)\le 0,\qquad&x\in\Omega^{'},\\
		w_{1}(x)\ge 0,w_{2}(x)\ge 0,\qquad&x\in \partial\Omega^{'},
	\end{cases}$
\end{center}
where
\begin{center}
	$c_{11}(x)=\frac{\underline{\lambda}g(-\underline{u}(x),-\underline{v}(x))-\underline{\lambda} g(-u(x),-\underline{v}(x))}{\underline{u}(x)-u(x)}$, $c_{12}=\frac{\underline{\lambda} g(-u(x),-\underline{v}(x))-\lambda g(-u(x),-v(x))}{\underline{v}(x)-v(x)}$
\end{center}
and 
\begin{center}
	$c_{21}(x)=\frac{\underline{\lambda}h(-\underline{u}(x),-\underline{v}(x))-\lambda h(-u(x),-\underline{v}(x))}{\underline{u}(x)-u(x)}$, $c_{22}(x)=\frac{\lambda h(-u(x),-\underline{v}(x))-\lambda h(-u(x),-v(x))}{\underline{v}(x)-v(x)}$
\end{center} 
By \cite{de_figueredo_monotonicity_1994}, for $\varepsilon$ small enough, we have $w_{1}(x)=\underline u(x)-u(x)\ge 0$ and $w_{2}(x)=\underline v(x)-v(x)\ge 0$ in $\Omega^{'}$. So $\mathcal{C}'$ is open relative to $\mathcal{C}$. Therefore, we get $\mathcal{C}'=\mathcal{C}$.

So for all $(u,v)\in \mathcal{C}$, we have $u(x)\le  \underline{u}(x)$ and $v(x)\le  \underline{v}(x)$. But $(+\infty,(0,0))\in \mathcal{C}$, $\underline{u}\le 0$ and $\underline{v}\le 0$, so $\underline{u}=\underline{v}\equiv0$ which is in contradiction with  the choose of nontrivial solution $(\underline{u},\underline{v})$. So when $\lambda<\lambda_{*}$, \cref{eq:SKS} has no solution.

{\bf Step 2 }.For $\lambda>\lambda_{*}$, \cref{eq:SKS} has at least two nontrivial admissible solutions.

Let $\widetilde{\lambda}=inf\left\{\lambda>0:\exists (u_{1},v_{1})\not\equiv (u_{2},v_{2})~s.t.~(\lambda,(u_{1},v_{1})),(\lambda,(u_{2},v_{2}))\in \mathcal{C}\right\}$. Suppose on the contrary, that $\widetilde{\lambda}>\lambda_{*}$. Take $\lambda'\in \left(\lambda_{*}, \widetilde{\lambda}\right]$such that there exists a unique $(u,v)$ such that $(\lambda',(u,v))\in \mathcal{C}$. Define
    \begin{center}
    	$\mathcal{C}''=\left\{(\lambda,(u,v))\in \mathcal{C}:\lambda\ge \lambda'\right\}$.
    \end{center}
Clearly, $\mathcal{C}''$ is connected and closed. Then let $\overline{\mathcal{C}}=\left\{(\lambda,(u,v))\in \mathcal{C}'':u\le  u''~and~v\le  v''\right\}$, where $(\lambda'',(u'',v''))$ is a nontrivial solution of (6.1) and $\lambda''<\lambda'$. So we can repeat the argument of Step 1 to show the \cref{eq:SKS} have no solution for $\lambda<\lambda'$, which is a contradiction.
\end{prf4}

\section{Acknowledgments}

\bibliographystyle{abbrv}
\bibliography{ref}

\end{document}